\documentclass{article}
\usepackage[utf8]{inputenc}
\usepackage[T1]{fontenc}

\usepackage{amsmath,amsfonts,amssymb,amsthm,xcolor}
\usepackage[english]{babel}
\usepackage{algorithm2e}
\usepackage{algorithmic}

\usepackage{microtype}
\usepackage{graphicx}
\usepackage{subfigure}
\usepackage{booktabs} 
\usepackage{hyperref}

\usepackage[scale=0.8]{geometry}

\newcommand{\yex}{y^{\text{extr\,}}}
\newcommand{\rex}{r^{\text{extr\,}}}

\newcommand{\gw}{\gamma_W}

\newtheorem{theorem}{Theorem}[section]
\newtheorem{proposition}[theorem]{Proposition}

\renewenvironment{proof}{\textbf{Proof.}}{\QED\bigskip}

\newcommand{\tr}{\text{Trace}}

 \numberwithin{dummy}{section}

\definecolor{ddarkbrown}{rgb}{0.5,0.2,0.05} \definecolor{bbluegray}{rgb}{0.05,0,0.5}

\newcommand{\BEAS}{\begin{eqnarray*}}
\newcommand{\EEAS}{\end{eqnarray*}}
\newcommand{\BEA}{\begin{eqnarray}}
\newcommand{\EEA}{\end{eqnarray}}
\newcommand{\BEQ}{\begin{equation}}
\newcommand{\EEQ}{\end{equation}}
\newcommand{\BIT}{\begin{itemize}}
\newcommand{\EIT}{\end{itemize}}
\newcommand{\BNUM}{\begin{enumerate}}
\newcommand{\ENUM}{\end{enumerate}}

\newcommand{\BA}{\begin{array}}
\newcommand{\EA}{\end{array}}

\newcommand{\QED}{~~\rule[-1pt]{6pt}{6pt}}

\newcommand{\argmin}{\mathop{\rm argmin}}

\title{Generalized Framework for Nonlinear Acceleration}
\author{Damien Scieur}
\date{Princeton University, Computer Science Department}

\usepackage[numbers,square]{natbib} 

\usepackage[parfill]{parskip}

\begin{document}

\maketitle

\begin{abstract}
Nonlinear acceleration algorithms improve the performance of iterative methods, such as gradient descent, using the information contained in past iterates. However, their efficiency is still not entirely understood even in the quadratic case. In this paper, we clarify the convergence analysis by giving general properties that share several classes of nonlinear acceleration: Anderson acceleration (and variants), quasi-Newton methods (such as Broyden Type-I or Type-II, SR1, DFP, and BFGS) and Krylov methods (Conjugate Gradient, MINRES, GMRES). In particular, we propose a generic family of algorithms that contains all the previous methods and prove its optimal rate of convergence when minimizing quadratic functions. We also propose multi-secants updates for the quasi-Newton methods listed above. We provide a Matlab code implementing the algorithm.
\end{abstract}

\section{Introduction}
Consider the simple fixed-point iteration
\[
    x_{i+1} = g(x_i)
\]
which produces a sequence of points $\{x_0,x_1,\ldots,x_N\}$. In most cases this converges to the fixed-point $x^*$,
\[
    x^* = g(x^*).
\]
This setting is quite generic, for example in the case of optimization $g$ can be a gradient step on an objective function $f$ and $x^*$ is its minimizer. The sequence $\{x_i\}$ converges to $x^*$ at a certain speed, but ideally, we would like the procedure to be as fast as possible. For example, in optimization the accelerated gradient method \cite{nesterov2013introductory}
\[
    \begin{cases}
        x_{i+1} & = y_i-h\nabla f(y_i) \\
        y_{i+1} & = (1+\beta) x_{i+1}-\beta x_i
    \end{cases}
\]
converges faster to the optimum $x^*$ than gradient method, provided good constants $h$ and $\beta$. In practice, those constants may be hard to estimate, especially $\beta$ which depends on the strong convexity parameter of the objective function, whose estimation is still a challenge \citep{fercoq2016restarting}.

Some nonlinear acceleration algorithms such as Anderson Acceleration share the same idea of Nesterov's acceleration. It combines the gradient step with a linear combination of previous iterates as follows,
\BEA
    x_{i} & = & g(y_{i-1}), \label{eq:family_algo}\\
    y_{i} & = & \textstyle \sum_{k=1}^{i} \alpha_k^{(i)} x_k, \nonumber
\EEA
where the vector $\alpha^{(i)}$ is function of the iterates, thus changing over time. As the coefficients $\alpha$ depend on the iterates $x_i$, the acceleration is thus called \textit{nonlinear}. Its main drawbacks is the lack of convergence guarantees, and in fact it has been showed that Anderson Acceleration is unstable when $g(x)$ is not a deterministic linear function \cite{scieur2016regularized}. The same paper proposes a regularized version of Anderson acceleration, whose rate of convergence is asymptotically optimal even in the presence of noise \citep{scieur2017nonlinear} or when the Jacobian of $g$ is non-symmetric \cite{scieur2018nonlinearb,bollapragada2018nonlinear}.

Other techniques such as Quasi-Newton methods schemes, popular in optimization, approximate the Newton step using the matrix $H \approx (\nabla^2 f(x_i))^{-1}$ as follow,
\[
    x_{i+1} = x_i - H\left( h  \nabla f(x_i)\right).
\]
This can be extended to fixed-point iteration by coupling a fixed-point step with a Quasi-Newton step,
\BEAS
    x_{i} & = & g(y_{i-1}), \\
    y_{i} & = & y_{i-1} - H(x_i-y_{i-1}).
\EEAS
Such matrix $H$ can be found using several formulas. The simplest ones are Broyden Type-I and Type-II updates \cite{broyden1965class}, and the most popular is certainly BFGS or of DFP \cite{nocedal2006nonlinear}. There also exists the symmetric rank-one update which has been rediscovered many time in many different fields.

Finally, we study Krylov subspace techniques such as the Conjugate Gradient method and GMRES \citep{saad1986gmres}.  These algorithms minimize some error function using a Krylov basis, usually updated with orthonormal vectors to ensure stability. Their primary usage is solving large systems of linear equations and optimizing quadratic functions.

The optimal convergence rate of Krylov methods is well-known when the fixed-point operator $g$ is a linear mapping, and works of \citep{scieur2018nonlinearb} show similar performance for Anderson Acceleration. For quasi-Newton methods, the results are less clear, even for quadratic objectives with two variables. For example,  DFP and BFGS algorithms may converge poorly without line-search \citep{powell1986bad}. 

When the function $g$ is nonlinear, it is unclear how fast those methods converge. In particular, the bad theoretical rates of convergence (if any) does not match the usual good numerical performance. The lack of robustness of nonlinear acceleration algorithms can explain this phenomenon since instability issues are known for some of them \citep{powell1977restart,johnson1988modified,scieur2016regularized}.

With recent result from \cite{scieur2016regularized,scieur2017nonlinear,scieur2018nonlinear}, it is now possible to have nonlinear acceleration techniques that achieve an asymptotically optimal rate of convergence even in the presence of stochastic noise. However, because the analysis of nonlinear acceleration methods is independent of each other, we unify the analysis to identify the central argument of nonlinear acceleration. 

Several results linked some acceleration methods to each other. For example \cite{fang2009two}  propose a general family of Broyden methods, including Type-I or Type-II updates, as well as Anderson mixing. \citet{walker2011anderson} show the link between Anderson and GMRES. However, the study does not include common schemes, such as BFGS and DFP.

\textbf{Contributions}. In this paper, we propose the Generalized Nonlinear Acceleration algorithm which mixes Anderson acceleration and quasi-Newton methods. In function of its parameters, it can produce the same steps than Anderson Acceleration, Broyden Type-I or Type-II, DFP, BFGS, SR-$k$ (symmetric rank $k$ update) or even conjugate gradients and GMRES. We give the proof of its (optimal) rate of convergence when applied to a linear function $g$, in the metric $\|\cdot\|_W = \|W^{1/2}\cdot\|_2$ where $W$ is positive definite. We derive the multi-secant updates for DFP and BFGS and extends the SR-1 to SR-$k$ updates, then analyze connections with quasi-Newton methods. We show equivalences between our algorithm and with multi-secant updates of the estimate of the Hessian. We also investigate the links with Krylov methods and propose another way to generalize CG for solving a nonlinear system of equations (or for minimizing non-quadratic functions). 

\textbf{Paper Organization.} TODO

\subsection{Notations and Assumptions} \label{sec:notations}

This paper studies a way to accelerate the convergence of the family of algorithms \eqref{eq:family_algo} when $g$ is linear, i.e.,
\BEA
    g(x) & = & G(x-x^*) + x^*.\nonumber
\EEA
Usually, such mapping is written as $Ax+b$ since $x^*$ is not explicitly known. However this notation is more convenient for our theoretical analysis. In particular, we study the following family of algorithm, that alternates between one fixed-point iteration and one linear combination step,
\BEA
    x_{i} & = & G(y_{i-1}-x^*) + x^*,  \label{eq:linear_algo}\\
    y_{i} & = & \textstyle \sum_{k=1}^{i} \alpha_k^{(i)} x_k.\nonumber
\EEA
Thorough this paper, we always assume that
\begin{itemize}
    \item $G$ is a symmetric definite positive matrix, whose spectrum is bounded by $0 \preceq G \preceq (1-\kappa) I$ with $\kappa < 1$. Usually, $\kappa$ is close to zero and often refers to be the inverse of a condition number.
    \item $\textstyle \sum_{k=1}^{i} \alpha_k^{(i)} = 1$, to have a consistent algorithm \cite{scieur2017integration}. It ensures $y_i = x^*$ when all $x_k$ are replaced by $x^*$.
    \item The last coefficient $\alpha_i^{(i)}$ is nonzero. Intuitively, if this coefficient is equal to zero, we waste the last call of $g$, making that iteration useless.
\end{itemize}
\paragraph{Polynomial Notation} \citet{scieur2018nonlinearb} shows that \eqref{eq:linear_algo} is equivalent to a sequence of polynomials,
\BEA
    x_{i} & = & G(y_{i-1}-x^*) + x^*,   \label{eq:poly_algo}\\
    y_{i} & = & p_i(G) (y_0-x^*)+x^*,\nonumber
\EEA
where $p_i(G)$ is a polynomial of degree \textit{exactly} $i$, whose coefficients sum to one.

\subsection{Residual and Rate of Convergence}
 We define the residual
\BEQ
    r_i = x_i-y_{i-1}. \label{eq:def_residual}
\EEQ
In this paper, we often refer to the link between the residual \eqref{eq:def_residual} and $y_i-x^*$. In particular,
\BEQ
    r_i = g(y_{i-1})-y_{i-1} = (G-I)(y_{i-1}-x^*). \label{eq:link_residual}
\EEQ
We can write this relation under the matrix form. Let the matrices, assumed to be \textbf{full column rank},
\BEA
    X   & = & [x_1,x_2,\ldots,x_N], \nonumber \\
    Y   & = & [y_0,y_1,\ldots,y_{N-1}], \label{eq:matrix_form}\\
    X^* & = & x^*\textbf{1}^T_N =  [x^*,\ldots,x^*] \;\;(\text{$N$ times}). \nonumber
\EEA
In this case, the relation \eqref{eq:link_residual} becomes
\BEQ
    R = (G-I)(Y-X^*). \label{eq:link_residual_matrix}
\EEQ

We now bound the performance of algorithm \eqref{eq:linear_algo} or equivalently \eqref{eq:poly_algo} in  the \textit{weighted Euclidean norm}
\BEQ
    \| v \|_W = v^TWv, \quad W\succ 0. \label{eq:def_metric}
\EEQ
Using \eqref{eq:poly_algo} and \eqref{eq:link_residual}, the norm \eqref{eq:def_metric} of $r_{i+1}$ can be bounded by
\BEA
    \| r_{i+1} \|_W \hspace{-2ex}& = &\hspace{-2ex}  \|(G-I)p_i(G)(y_0-x^*)\|_W = \| p_i(G) r_1\|_W \nonumber\\
    & \leq & \hspace{-2ex}\| p_i(G) \|_2 \|r_1\|_W. \label{eq:norm_residual}
\EEA
This means the performance of the algorithm \eqref{eq:linear_algo} can be summarized by the study of $\| p_i(G) \|_2$. Ideally, we would like to find the smallest polynomial to ensure fast convergence. In this paper, we propose the Generalized Nonlinear Acceleration algorithm that finds the best polynomial in \eqref{eq:norm_residual} at each iteration to ensure good convergence speed.

\section{Generalized Nonlinear Acceleration}
\label{sec:generic_acc_method}

The Generalized Nonlinear Acceleration Algorithm \ref{algo:gna} combines the ideas from Anderson Acceleration and quasi-Newton methods. In short, it combines linearly iterates that have been refined by a preconditioner $P$. This aims to minimizes the residual \eqref{eq:def_residual} in the weighted Euclidean norm\eqref{eq:def_metric} defined by the weights matrix $W\succ 0$.

\begin{algorithm}[htb]
   \caption{Generalized Nonlinear Acceleration}
    \label{algo:gna}
\begin{algorithmic}
   \STATE {\bfseries Data:} Matrices $X$ and $Y$ of size $d\times N$.
   \STATE {\bfseries Parameters:} Weight matrix $W\succ 0$, Preconditioner $P$.\\
   \hrulefill
   \STATE \textbf{1.} Compute matrix of residual $R = X-Y$.
   \STATE \textbf{2.} Solve
   \BEQ
        \gamma_W = \frac{(R^TWR)^{-1}\textbf{1}_N}{\textbf{1}_N^T(R^TWR)^{-1}\textbf{1}_N} =  \argmin_{\gamma:\textbf{1}^T\gamma = 1} \| R\gamma \|_W. \label{eq:gw}
    \EEQ
    \STATE \textbf{3.} Perform the extrapolation
    \BEQ
        \yex = (Y-PR)\gamma_W. \label{eq:gna_step}
    \EEQ
\end{algorithmic}
\end{algorithm}

In Algorithm \ref{algo:gna}, the parameters $P$ and $W$ are user-defined. In the next section, we discuss two standard way to choose $W$. Later in the paper, we show that the choice of $P$ algorithm \eqref{algo:gna} can produce steps that are identical to existing nonlinear acceleration algorithms. For now, we can consider for simplicity that $P$ is the scaled identity $\beta I$.

The coefficients $\gw$ \eqref{eq:gw}, when $W=I$, correspond to Lemma 2.4 in \citep{scieur2016regularized}. With a minimal adaptation of the proof, this extends to arbitrary $W\succ 0$. We now quickly discuss of two "classical" choices for the weight matrix $W$.

\subsection{Choice of \texorpdfstring{$W$}{W}}\label{sec:choice_w}

We briefly discuss two possible choices of the weight matrix $W$. The first one is the simple case where $W=I$, the second one is $W=(G-I)^{-1}$.

\subsubsection{Case where \texorpdfstring{$W=I$}{W=I}}
In the first case, $W=I$ simply recover the classical Anderson Acceleration \citep{anderson1965iterative}, when $P=\beta I$ (where $\beta \neq 0$ is a scalar). This algorithm is known to minimize the residual of the extrapolation, achieving in the worst case an optimal rate of convergence \citep{scieur2016regularized}.

\subsubsection{Case where \texorpdfstring{"$W=(G-I)^{-1}$"}{"W=(G-I){-1}"}}
The case when $W=(G-I)^{-1}$ looks impossible at first sight, as it requires the inversion of $(G-I)$. However, the next proposition shows we do not need to use it explicitly.

\begin{proposition} \label{prop:sol_good_anderson}
Let any symmetric positive definite matrix $W$ that satisfies
\BEQ
    WR=Y-X^*. \label{eq:cond_inv_w}
\EEQ
For example, $W=(G-I)^{-1}$. Then, the coefficients $\gw$ defined in Algorithm \ref{algo:gna} can be computed using the formula
\BEQ
    \gw = \frac{(Y^TR)^{-1}\textbf{1}_N}{\textbf{1}_N^T(Y^TR)^{-1}\textbf{1}_N}.
\EEQ
\end{proposition}
The proof can be found in Appendix \ref{prop:sol_good_anderson_proof}.

\section{Computational Complexity of GNA}
\label{sec:complexity_appendix}

The next proposition gives the computational complexity of Algorithm \ref{algo:gna} when matrices are updated properly. In short, for standard choices of $W$ and $P$, Algorithm \ref{algo:gna} take $O(Nd)$ operation per call. The memory size $N$ is small in practice, thus computing $\yex$ is as costly as the step \eqref{eq:linear_algo}.

\begin{proposition}
Assume we already have computed, using $N$ iterates from \eqref{eq:linear_algo},
\[
    WR, \quad (R^TWR)^{-1}\textbf{1}, \quad PR.
\]
Let $y^+$ and $r^+$ be the new iterate and residual. The update
\[
    WR^+, \quad ((R^+)^TWR^+)^{-1}\textbf{1}, \quad PR^+.
\]
costs at most $O_W+O_P+O(Nd+N^2)$ operations, where $O_W$ and $O_P$ are upper bounds on the number of operations for $Wr^{+}$  and $PR^+$.
\end{proposition}
\begin{proof}

First, we update the matrix $R^TWR$ to $[R,r^+]^TW[R,r^+]$. After expansion,
\[
    [R,r^+]^TW[R,r^+] = \begin{bmatrix}
        R^TWR & R^TWr^+ \\
        (r^+)^TWR & (r^+)^TWr^+ 
    \end{bmatrix}
\]
Since $WR$ and $R^TWR$ are known, the complexity of the update is equal to $O(O_W)$ (computation of the bottom-right element) plus $O(Nd)$.

The second step produces the coefficients $\gw$ by solving the system of equations $([R,r^+]^TW[R,r^+])^{-1}\textbf{1}_N$. The size of the system grows in $N$, which means $O(N^{3})$ operation to solve it. However, we update the matrix $R^TWR$ with a rank two matrix. With the Woodbury matrix identity (Appendix \ref{sec:woodbury}) it is possible to update the previous solution $\gw$, reducing the complexity to $O(N^2)$.

Finally, the last step consists in forming the extrapolation. First, it takes $O(dN)$ operation to form $Y\gw$ and $R\gw$. Then, we apply the preconditioner $P$ to $R\gw$, which takes $O(O_p)$ operations.

The total complexity is thus bounded by
\[
    O_W+O_P+O(Nd+N^2).
\]
\end{proof}

In practice, $N$ is much smaller than $d$, and usually $N\in [5,20]$ in most application. In addition, the matrix $W$ is usually used implicitly (see section \ref{sec:choice_w}), so $O_W=O(1)$. Finally, the preconditioner $P$ corresponds to the sum of the identity matrix and a rank $N$ matrix in most nonlinear acceleration algorithms, thus $O_P=O(N^2d)$. Under those conditions, the computational complexity of GNA is bounded by $O(N^2d)$ where $N^2$ is small compared to $d$. This property is desirable as it is \textit{as expansive as} computing a new residual $r^+$.

\section{Rate of Convergence}
\subsection{Chebyshev Acceleration}

We have seen that the norm of the polynomial \eqref{eq:norm_residual} quantifies the rate of convergence of algorithm \eqref{eq:poly_algo},
\[
    \| r_{i+1} \|_W \leq \|p_i(G)\|_2 \| r_1\|_W,
\]
where $p_i$ is a polynomial of degree exactly equal to $i$ whose coefficients sums to one. The best polynomial of this class is a rescaled Chebyshev polynomial \cite{golub1961chebyshev}, which achieves the optimal rate
\BEQ
    \textstyle \| r_{i+1} \|_W \leq \frac{\xi^i}{1+\xi^{2i}} \|r_1\|_W, \quad \xi = \frac{1-\sqrt{\kappa}}{1+\sqrt{\kappa}} \label{eq:cheby_rate}
\EEQ
where $\|G\|_2\leq 1-\kappa$. This requires the knowledge of $\kappa$, usually referring to the inverse of a condition number and thus unknown in practice. In addition, Chebyshev acceleration is not adaptive to the initial point $r_1$. The next section shows that Algorithm \ref{algo:gna} achieves the same bound without the knowledge of $\kappa$.

\subsection{Optimal Rate of Convergence of (offline) GNA}

The next Theorem shows Algorithm \ref{algo:gna} implicitly solves
\[
    \textstyle \min_p \| p(G) r_1 \|_W \quad \text{s.t.} \,\, p(1) = 1,\,\, \deg(p) \leq N-1.
\]
This strategy is optimal, in the sense that it recovers \eqref{eq:cheby_rate} up to a constant that depends on the preconditionner $P$.

\begin{theorem}\label{thm:rate_conv}
Let $\{(x_1,y_0),\ldots, (x_N,y_{N-1})\}$ be pairs generated by \eqref{eq:linear_algo} (or equivalently \eqref{eq:poly_algo}) where the matrix $G$ is symmetric and $\|G\|\leq 1-\kappa$. Let $\yex$ be the output of Algorithm \ref{algo:gna}. Consider the residual of the extrapolation
\[
    \rex = g(\yex)-\yex.
\]
Then, the norm $\|\rex\|_W$ is bounded by
\BEA
    \hspace{-3ex}\| \rex \|_{W} \hspace{-2ex} & \leq & \hspace{-2ex} \| I-(G-I)P \|_2 \hspace{-1ex} \min_{\substack{ \deg(p)\leq N-1\\ p(1)=1}} \|p(G)r_1\|_W \label{eq:opti_polynomial} \\
    & \leq & \hspace{-2ex} \| I-(G-I)P \|_2 \frac{\xi^{N-1}}{1+\xi^{2(N-1)}} \| r_1 \|_{W} \label{eq:rate_gna}
\EEA
where $\xi$ is defined in \eqref{eq:cheby_rate} and $W$ is a positive definite. In addition, after at most $d$ steps, $\rex = 0$.
\end{theorem}
The proof can be found in Appendix \ref{thm:rate_conv_proof}. This theorem shows that it is possible to ensure the optimal rate of convergence by post-processing the iterates using Algorithm \eqref{algo:gna} for any norm defined by $W$. Also, a finite number of iteration are required to reach $x^*$. 

Intuitively, it would be more efficient to inject at each iteration the output of GNA. The next section studies this strategy and gives a sufficient condition to ensure an optimal rate of convergence.

\subsection{Online GNA}

As discussed in the previous section, instead of computing $\yex$ on the side, we can inject the extrapolated point directly in \eqref{eq:linear_algo} as following,
\BEA
    x_{i} & = & G(y_{i-1}-x^*) + x^*, \label{eq:online_algo}\\
    y_{i} & = & \yex.\nonumber
\EEA
However, it is not clear from Algorithm \eqref{algo:gna} that $\yex$ satisfies the special structure \eqref{eq:poly_algo}, in particular we need to ensure that the extrapolation can be written as a polynomial in the matrix $G$, whose degree is exactly equal to $i-1$ and its coefficients sums to one. The next proposition gives a simple condition on $P$ to ensure those two properties.
\begin{proposition} \label{prop:online_accel}
Consider the output of Algorithm \ref{algo:gna} after $N$ iterations of \eqref{eq:poly_algo}. If we have, for $\beta \neq 0$,
\[
    \yex = (Y-PR)\gamma_W = Y\tilde\gamma + \beta R\gamma_{\tilde W},
\]
where $\tilde\gamma$ is an arbitrary vector whose coefficients sum to one, and $\tilde W$ a positive definite matrix, then $\yex - x^*= p_{N-1}(G)(y_0-x^*)$, where $p_N$ is a polynomial of degree $N-1$ whose coefficients sum to one.
\end{proposition}

The proof can be found in Appendix \ref{prop:online_accel_proof}. In short, this proposition shows that whatever the coefficients that combine the $y_i$'s, as long as we use some "optimal" coefficients for combining the residuals, the extrapolation satisfies the structure \eqref{eq:poly_algo}. By using this argument recursively, \eqref{eq:online_algo} produces iterates with the same structure than \eqref{eq:poly_algo}. In combination with Theorem \ref{thm:rate_conv}, this proves the optimality of the rate of convergence of online Generalized Nonlinear Acceleration algorithm.

It remains to fix the value of $P$. The next sections show some choices of the preconditioner that correspond to existing nonlinear acceleration algorithms.

\section{Connections with Anderson Acceleration}
We begin with nonlinear acceleration algorithms that form a polynomial that minimizes \eqref{eq:norm_residual} at each iteration. The link between GNA and this class of algorithm is quite straightforward as they share the same ideas.

The Anderson Acceleration \citep{anderson1965iterative}, Minimal Polynomial Method \cite{cabay1976polynomial} and Mesina Method \cite{mesi77} are different variant of the same algorithm. They solve, at each iteration,
\[
    \gamma_{\text{Anderson}} = \argmin_{\gamma:\textbf{1}^T\gamma = 1} \|R\gamma\|_2.
\]
Anderson Acceleration and Mesina method use the same formula, while MPE method uses a slightly different approach giving the same result. Then, Anderson acceleration uses $\gamma_{\text{Anderson}}$ to combines the previous iterates using a so-called \textit{mixing} parameter $\beta \neq 0$ as following,
\[
    \yex = (Y-\beta R) \gamma_{\text{Anderson}}.
\]
Clearly, this corresponds to a special case of Algorithm \ref{algo:gna} where $W=I$ and $P=\beta I$.

Recently, \citet{fang2009two} introduce the type-I (or "Good") Anderson Acceleration. They used the Broyden \mbox{Type-I} update to derive the formula
\BEQ
    \textstyle \gamma_{\text{Good Anderson}} = \frac{(Y^TR)^{-1}\textbf{1}}{\textbf{1}^T(Y^TR)^{-1}\textbf{1}}.\label{eq:formula_good_anderson}
\EEQ
Then, it combines the previous iterates as follow,
\[
    \yex = (Y-\beta R) \gamma_{\text{Good Anderson}}.
\]
This formula comes from the analogy between Anderson and Broyden methods, so it is unclear why \eqref{eq:formula_good_anderson} is a good candidate for minimizing the norm of the residual \eqref{eq:norm_residual}. Besides, the rate of convergence was not specified.

With our results, we see that equation \eqref{eq:formula_good_anderson} corresponds to \eqref{eq:gw} when $W=(G-I)^{-1}$. This means equation \eqref{eq:formula_good_anderson} is the solution of \eqref{eq:gw}. With the application of Theorem \ref{thm:rate_conv}, we directly obtain the rate of convergence of the Good Anderson Acceleration. Indeed, using GNA with $W=(G-I)^{-1}$ and $P=\beta I$ produces the same extrapolated point, thus leading to an optimal rate of convergence.

\section{Connections with Quasi-Newton (qN)}

We quickly introduce the idea of quasi-Newton methods. First, consider Newton methods step
\[
    y_{\text{Newton}} = y_{N-1}-(G-I)^{-1}r_N.
\]
Using equation \eqref{eq:link_residual}, we have $y_{Newton} = x^*$. However it needs $O(d^3)$ operations to invert the matrix $(G-I)$. 

Unlike the Newton method, the qN step is cheaper but does not converge in one iteration. It uses an approximation $H \approx (G-I)^{-1}$ and performs
\BEQ
    y_{\text{q-Newton}} = y_{N-1}-Hr_N, \label{eq:qn_step}
\EEQ
where the matrix $H$ satisfies \textit{secant equations}
\[
    H(r_{i+1}-r_{i}) = (y_i-y_{i-1}) \quad \forall i=1...N-1.
\]
In the matrix form, this gives
\BEQ
    HRC=YC, \quad \textbf{1}^TC = 0, \;\; \textbf{rank}(C) = N-1.\label{eq:secant_equation}
\EEQ
For example, $C$ can be formed using $N-1$ different vectors $[\ldots 0,1,-1,0,\ldots]^T$. Indeed, the secant equation still holds if $H$ is replaced by $(G-I)$. Some qN methods estimates instead the matrix $J\approx G-I$ then take its inverse $H=J^{-1}$. In this case, $J$ follows
\BEQ
    RC=JYC, \quad \textbf{1}^TC = 0, \;\; \textbf{rank}(C) = N-1. \label{eq:secant_equation_2}
\EEQ
Any $H$ (or $J^{-1}$) that satisfies the secant equation \eqref{eq:secant_equation} is a candidate for the qN step. There exist several strategies to choose one particular matrix. In the next section, we introduce the generalized quasi-Newton step, then we investigate further more classical schemes.

\subsection{Generalized qN Method}
We introduce here the Generalized qN scheme, which performs the extrapolation
\BEQ
    y_{\text{Generalized qN}} = (Y-HR)\gamma, \quad \forall\gamma^T\textbf{1}=1,\label{eq:generalized_qn}
\EEQ
where $\gamma$ is an arbitrary vector whose entries sum to one and $H$ satisfies the secant equation \eqref{eq:secant_equation}. Proposition \ref{prop:invariance_gamma} shows that \eqref{eq:generalized_qn} is invariant in the choice of $\gamma$. For illustration, classical schemes use $\gamma = [0,\ldots,0,1]$. 

The solution of the secant equation \eqref{eq:secant_equation} is given by
\BEQ
    H = YC(RC)^{\dagger} + H_0RC(RC)^{\dagger}, \label{eq:general_solution}
\EEQ
where $(RC)^{\dagger}$ is the generalized inverse of $RC$ and $H_0$ is arbitrary (it can be viewed as the initialization). There exist a large number of generalized pseudo-inverses, and choosing one of them corresponds to one qN method.

Some qN schemes minimize the distance between $H$ and $H_0$ in, for example, the \textit{weighted Frobenius norm}
\BEQ
    \|A\|_M = \tr(M^{1/2} A^T M A M^{1/2}), \quad M\succ 0. \label{eq:weighted_frobenius}
\EEQ
Other methods impose in addition some constraints on the matrix $H$, for example symmetry.

Despite the different approach of qN and GNA, writing the qN method under the form of \eqref{eq:generalized_qn} makes the links straightforward: \eqref{eq:generalized_qn} is equivalent to run Algorithm \ref{algo:gna} with $P=H$ and $W$ arbitrary. Using Theorem \ref{thm:rate_conv}, this means that using \textit{any} matrix that satisfies the secant equation \eqref{eq:secant_equation} in the qN step \eqref{eq:generalized_qn} lead to an optimal rate of convergence.

However, it is unclear if a given qN method meets the requirements of Proposition \ref{prop:online_accel}. In addition, the computation of \eqref{eq:general_solution} may be complicated in some cases. Finally, it is unclear how bad is the constant factor in Theorem \ref{thm:rate_conv}.

In the next section, we show that specific choices of $W$ (and thus $\gw$) simplify the term $HR\gamma$ in \eqref{eq:generalized_qn} for some known qN schemes. It leads to more explicit update rules that meet the assumptions of Proposition \ref{prop:online_accel}.

\subsection{Broyden Methods}

\subsubsection{(Multi-Secant) Type-I Broyden}

The Type-I Broyden method estimates the matrix $G-I$ first, then invert it. Among all the possible solutions of \eqref{eq:secant_equation_2} it chooses the closest to the initialization $J_0$ in the Frobenius norm. Here, we consider the general case in the weighted Frobenius norm, i.e.,
\[
    \textstyle J = \argmin_J \| J-J_0 \|_M,\quad JYC=RC,
\]
then $H=J^{-1}$ using the Woodbury identity (See Appendix \ref{sec:woodbury}). The explicit formula for $H$ is derived in Appendix \ref{sec:good_broyden}. After simplification, injecting $H$ in \eqref{eq:generalized_qn} gives
\BEA
    & & y_{\text{Type-I Broyden}} = (Y-J_0^{-1}R)\gamma_{\tilde M} \label{eq:broyden_step} \\
    & & \text{Where} \;\;\;  \tilde M = (G-I)^{-1}M^{-1}J_0^{-1}. \label{eq:variable_metric}
\EEA
This corresponds to Algorithm \ref{algo:gna} with parameters $P = J_0^{-1}$ and $W = \tilde M$. At this state, this is unclear if $\tilde M$ is symmetric definite definite, so Theorem \ref{thm:rate_conv} does not apply in this situation. If we consider the simple case where $M=I$ and $J_0^{-1} = \beta I$ with $\beta \neq 0$, \eqref{eq:broyden_step} simplifies into
\BEQ
    y_{\text{Type-I Broyden}} = (Y-\beta R)\gamma_{(G-I)^{-1}}. \label{eq:broyden_type1_step}
\EEQ
This corresponds to Algorithm \ref{algo:gna} with $P = \beta I$ and $ W = \beta (G-I)^{-1}$. With these parameters Theorem \eqref{thm:rate_conv} shows that the Broyden Type-I method is an optimal algorithm for minimizing the residual in the norm $\|\cdot\|_{(G-I)^{-1}}$. In addition, the step meets the structure of Proposition \ref{eq:online_algo}, so online acceleration is possible.

\subsubsection{(Multi-Secant) Type-II Broyden}
The Type-II Broyden method estimates directly $(G-I)^{-1}$ by choosing $H$ as close as possible to the initialization $H_0$ in the Frobenius norm. Again, we give here the update in the weighted Frobenius norm \eqref{eq:weighted_frobenius} (see Appendix \ref{sec:bad_broyden} for more details). We have
\[
    \textstyle H = \argmin_H \|H-H_0\|_{M^{-1}} \quad s.t. \;\; HRC=YC.
\]
After simplification, injecting $H$ in \eqref{eq:generalized_qn} gives
\[
    y_{\text{Type-II Broyden}} = (Y-H_0R)\gamma_M.
\]
This time, since $M$ is symmetric positive definite, Theorem \ref{thm:rate_conv} directly applies, but this is not the case with Proposition \eqref{eq:online_algo}. However if we assume $H_0 = \beta I$ with $\beta \neq 0$, online acceleration is possible since
\BEQ
    y_{\text{Type-II Broyden}} = (Y-\beta R)\gamma_{I}. \label{eq:broyden_type2_step}
\EEQ

\subsubsection{Observations}

Usually, the matrices $J_0$ and $H_0$ correspond to the previous estimate and are updated with rank-one matrices. However, for the  Type-I method, it is unclear if the method has an optimal rate of convergence, as \eqref{eq:variable_metric} may not be symmetric positive definite. Thus Theorem \ref{thm:rate_conv} does not apply.

When the initialization is a scaled identity matrix, The Broyden Type-I \eqref{eq:broyden_type1_step} and Broyden Type-II \eqref{eq:broyden_type2_step} steps are equivalent up to the weight matrix $M$. For example, the Broyden Type-2 updates with $M=(G-I)^{-1}$ is identical to Broyden Type-I with $M=I$. In this case, they are also equivalent to Anderson Acceleration.

\subsection{Symmetric Methods}

We introduce the multi-secant symmetric updates. This setting was studied by \cite{schnabel1983quasi} for DFP and BFGS (not SR-$k$), but not for the general weighted Frobenius norm.  It requires that \eqref{eq:secant_equation_2} admits a symmetric solution, or equivalently that  \citep{don1987symmetric,hua1990symmetric}
\[
    (YC)^TRC = (RC)^T(YC).
\]
This may not be the true when $g$ in \eqref{eq:family_algo} is nonlinear, but in our setting \eqref{eq:linear_algo} the symmetric solution exists since
\[
    (YC)^TRC = (YC)^T(G-I)YC =(RC)^T(YC).
\]

\subsubsection{(Multi-Secant) DFP}

The DFP Algorithm \cite{davidon1991variable,fletcher2013practical} finds the closest \textit{symmetric} matrix to $J_0$ in the weighted Frobenius norm defined by the matrix $(G-I)^{-1}$ that solves \eqref{eq:secant_equation_2}. In Appendix \ref{sec:dfp} we consider the general case 
\BEQ
    \textstyle J = \argmin_J \| J-J_0 \|_M,\;\; JYC=RC,\,\, J=J^T, \label{eq:optim_dfp}
\EEQ
where $M$ is arbitrary. The explicit formula (see Appendix \ref{sec:dfp}) is more complicated than Broyden updates, so we directly jump to the standard and simpler case where
\[
    M=(G-I)^{-1}, \quad J_0^{-1} = \beta I.
\]
Injecting the solution of \eqref{eq:optim_dfp} in the  qN step \eqref{eq:generalized_qn} gives
\begin{equation*}
    y_{\text{DFP}} = Y\gamma_I-\left(\beta I + YC \left( (YC)^TRC \right)^{-1}(YC)^T\right)R\gamma_I 
\end{equation*}
This is equivalent to Algorithm \ref{algo:gna} with parameters $W=I$ and $P=\beta I + YC \left( (YC)^TRC \right)^{-1}(YC)^T$. This step can be written in the form of Propopsition \eqref{prop:online_accel},
\[
    Y\underbrace{\big((I-C ( (YC)^TRC )^{-1}(YC)^TR\big)\gamma_I}_{=\tilde \gamma} + \beta R\gamma_I.
\]
By Proposition \ref{prop:online_accel} and Theorem \ref{thm:rate_conv}, online acceleration is possible if $\beta \neq 0$ with an optimal convergence rate. 

\subsubsection{(Multi-Secant) BFGS}
We now introduce the BFGS algorithm. The idea is similar to the DFP formula, but it updates the approximation \mbox{$(G-I)^{-1}$} directly. Again, solving the general case
\[
    \argmin_H \| H-H_0 \|_{M^{-1}} \quad s.t. \;\; HRC=YC,\;\; H=H^T
\]
leads to complicated formula (see Appendix \ref{sec:bfgs} for details). We directly jump to the standard case where
\[
    M=(G-I)^{-1},\quad H_0=\beta I.
\]
Injecting $H$ in the generalized qN step \eqref{eq:generalized_qn}gives
\BEAS
    y_{\text{BFGS}} & = & Y\gamma_{(G-I)^{-1}} - \beta R\gamma_{(G-I)^{-1}}\\
    & + & \beta YC\big( (YC)^TRC \big)^{-1}(RC)^T R\gamma_{(G-I)^{-1}}.
\EEAS
This corresponds to Algorithm \ref{algo:gna} with parameters 
\[
    W=(G-I)^{-1},\,\, P=\beta \big(I-YC\big( (YC)^TRC \big)^{-1}(RC)^T\big).
\]
With the initialization $H_0=\beta I$ we can use online acceleration (Proposition \ref{prop:online_accel}). Unlike DFP, the usage of $\beta$ here is similar to \eqref{eq:broyden_type1_step} and \eqref{eq:broyden_type2_step}.

\subsubsection{SR-\texorpdfstring{$1$}{1} and SR-\texorpdfstring{$k$}{k}}
We now introduce the symmetric rank $k$ (SR-$k$) update, whose SR-$1$ is a particular member. The method starts with an initialization $H_0$, then looks to a symmetric, low-rank update $H-H_0$ such that $H$ satisfies the secant equation \eqref{eq:secant_equation}. More formally, SR-$k$ solves
\[
    H = \min \text{rank}(H-H_0) \quad \text{s.t.}\;\; HRC=YC,\;\; H=H^T.
\]
Its explicit solution is given  in Appendix \ref{sec:srk},and injecting this matrix in the generalized qN step \eqref{eq:generalized_qn} gives
\BEQ
    y_{\text{SR-}k} = (Y-H_0R)\gamma_{(G-I)^{-1}-H_0}. \label{eq:srk_step}
\EEQ
The vector $\gamma_{(G-I)^{-1}-H_0}$ is computed using the formula
\[
    \gamma_{(G-I)^{-1}-H_0} = \frac{ \left(Y^TR-R^TH_0R\right)^{-1}\textbf{1}}{\textbf{1}^T\left(Y^TR-R^TH_0R\right)^{-1}\textbf{1}}.
\]
We have here a direct correspondence between the SR-$k$ update and Algorithm \ref{algo:gna} with
\[
    P = H_0,\quad W = (G-I)^{-1}-H_0.
\]
Under the condition that $(G-I)^{-1} \succ H_0$, Theorem \eqref{thm:rate_conv} applies, so the SR-$k$ updates has an optimal rate of convergence. If $H_0=\beta I$, by proposition \ref{prop:online_accel} the SR-$k$ algorithm \eqref{eq:srk_step} can be used online.

\subsection{Comparison with Previous Work} \label{sec:single_secant}

We showed, with Theorem \ref{thm:rate_conv}, that qN methods with multi-secant updates converge with a rate similar to Chebyshev acceleration \eqref{eq:cheby_rate}. In addition, after at most $d$ iterations they reach exactly $x^*$. In practice, quasi-Newton methods are used with a single secant update, so our convergence result may not apply, as \eqref{eq:secant_equation} may not hold for all secants.

There exist some convergence results for single secant update of some quasi-Newton methods. For example, \citet{gay1979some} shows that Broyden Type-I and Type-II update converges to $x^*$ after $2d$ steps. \citet{nocedal2006nonlinear} show that BFGS and DFP with \textit{with exact line-search} converges after $d$ steps, but it does not hold for inexact line search or unitary step size. Moreover, the rate of convergence for the first $N$ iterations is still unknown, and in fact can be particularly bad even for a quadratic problem with two variables \citep{powell1986bad}.

Surprisingly, the SR-1 method converges after $d$ iterations \cite{nocedal2006nonlinear} because it satisfies all previous secant equations. Here, with Theorem \ref{thm:rate_conv} we refine the result by specifying its rate on quadratics for any $N<d$.

\textit{A priori}, there is no reason to use single-secant updates for optimizing quadratics as they are as costly as the multi-secant update, and their theoretical properties are apparently weaker. This may not be true numerically, so we compare the two approaches in Appendix \ref{sec:num_experiments}.

\section{Connections with Krylov Methods}

We investigate connections with Krylov methods, used for solving linear systems $Ax=b$ where $A$ is positive definite, or equivalently minimizing a quadratic function. In the case where $g(x)$ is a linear function \eqref{eq:linear_algo}, we can also use Krylov methods to find an extrapolation $\yex$.

Krylov methods start with an approximation $x_0$, associated to the residual $r_1$, then generate the \textit{Krylov subspace}
\BEA
    \mathcal{K}_{N} & = & \text{span}\{ \underbrace{[r_1,Gr_1,G^2r_1,\ldots, G^{N-1}r_1]}_{=K_N}\}  \nonumber \\
    & = & \{q(G)r_1: \deg(q) \leq N-1\}. \label{eq:krylov_subspace}
\EEA
As the \textit{Krylov matrix} $K_N$ can be ill-conditioned in practice, Krylov subspace methods maintain an orthonormal basis to $\mathcal{K}_N$ to ensure better numerical stability.

After creating the subspace, Krylov methods minimize some error function $e(x)$, like the norm of the residual, under the constraint that the next iterate belongs to $x_0+\mathcal{K}_N$,
\BEQ
    \yex = \min_{x\in x_0 + \mathcal{K}_N} e(x). \label{eq:krylov_method}
\EEQ
Different error functions lead to different Krylov methods.

\subsection{MINRES and GMRES}
In the case of GMRES (or equivalently MINRES, since we work with symmetric matrices), the iterations are \textit{not} coupled. The algorithm creates a basis $K_N$ of $\mathcal{K}_N$ using Arnoldi (MINRES, \cite{paige1975solution}) or Lancoz (GMRES, \cite{saad1986gmres}) then computes
\[
    y_{\text{GMRES}} = \argmin_{x\in x_0 + \mathcal{K}_{N-1}} \| x-g(x) \|_2.
\]
We see the iterates belongs to $\mathcal{K}_{N-1}$ rather than $\mathcal{K}_{N}$. This can be explained by the fact that the residual of $y_{\text{GMRES}}$ belongs to $\mathcal{K}_{N}$.
This corresponds to \eqref{eq:krylov_method} with $e(x)=\|\cdot\|_2$. In Appendix \ref{sec:gmres}, we show that
\[
    y_{\text{GMRES}} \hspace{-0.5ex}=\hspace{-0.5ex} x^* + p^*(G)(x_0-x^*), \;\; p^* \hspace{-0.5ex}=\hspace{-1.5ex} \argmin_{\substack{\deg(p)\leq N-1,\\p(1)=1.}}\hspace{-1.5ex} \|p(G)r_1\|_2.
\]
Even if $y_{\text{GMRES}}$ belongs to $x_0+\mathcal{K}_{N-1}$, the algorithm uses $\mathcal{K}_{N}$ to compute the polynomial $p^*$. We can cast GMRES into an instance of GNA using $P=0$, $W = I$. Indeed,
\[
    y_{\text{GNA}} = Y\gamma_I = x^* + (Y-X^*)\gamma_I=  x^* + (G-I)^{-1}R\gamma_I
\]
Appendix \ref{sec:gmres} shows that $R\gamma_I = p^*(G)r_1$, thus
\[
    y_{\text{GNA}} = y_{\text{GMRES}} \;\; \text{when } W=I,\;\; P=0.
\]
Thus, GNA and GMRES / MINRES are equivalent. Since $P=0$, by Proposition \ref{prop:online_accel}, these algorithms are not suitable for online acceleration.

\subsection{Conjugate Gradient}

The Conjugate Gradient in an online acceleration algorithm, with is usually represented under the form of a two-step recurrence. Conceptually, CG builds a new iterate whose residual is orthogonal to the search space \eqref{eq:krylov_subspace}. This can be written under the form of \eqref{eq:krylov_method},
\[
    y_{\text{CG}} = \argmin_{y \in x_0 + \mathcal{K}_{N}} \|(K_N)^T(G-I)(y-x^*)\|_2
\]
Appendix \ref{sec:cg} shows that $y_{\text{CG}}$ is obtained using  GNA with 
\BEQ
    P = R(R^T(G-I)R)^{-1}R^T \quad \text{and} \quad W \text{ arbitrary}. \label{eq:gna_cg_step}
\EEQ
Here we see that we need to multiply $R$ with $G-I$, which implies two call to $g(x)$ per iteration. Appendix \ref{sec:cg} shows how to avoid this problem using the linearity of $g$ by deducing the residual of $y_{CG}$. With this strategy, the matrix $R^TR$ is diagonal and $R^T(G-I)R$ tri-diagonal. The solution can thus be updated using a two-step recurrence.

\textbf{Conjugate Broyden Method} Using \eqref{eq:gna_cg_step} with $W=I$, the iterate of GNA simplifies into
\BEQ
    y_{\text{GNA}} = Y\gamma_{I} + \beta^* R\gamma_{G-I}, \;\; \textstyle \beta^* = \frac{\textbf{1}^T(R^T(G-I)R)^{-1}\textbf{1}}{\textbf{1}^T(R^TR)^{-1}\textbf{1}}. \label{eq:conjugate_broyden}
\EEQ
In Appendix \ref{sec:cg}, we show that $\beta^*$ can be found by solving
\[
    \textstyle \beta^* = \argmin_{\beta,z}  f(z), \;\; z =  Y\gamma_{I} + \beta R\gamma_{G-I},
\]
where $f(z)$ is the quadratic objective function whose gradient step is represented by $g$. This algorithm a possible extension of nonlinear conjugate gradient. The computation of $\gamma_{G-I}$ requires one Hessian-vector operation, or an approximation of $G-I$. Then, one line-search on the objective function $f$ is needed to compute $\beta$. 

Compared to other nonlinear CG algorithm, it requires more memory ($O(N)$ rather than the two previous iterations), but this looks more stable as it does not relies on too much on the linearity of $g$. We present some numerical experiments on nonlinear functions in Appendix \ref{sec:num_experiments}.

\section{Numerical Experiments}

We briefly compare several instances of GNA for optimizing a quadratic function. We minimize the linear regression
\[
    \min_x \frac{1}{2} \left(\| Ax-b \|^2_2 + \lambda \|x\|_2^2\right)
\]
using the fixed-step gradient method for strongly convex functions \cite{nesterov2013introductory}, combined with GNA using online acceleration \eqref{eq:online_algo}.

We use data sets \texttt{Rand} and \texttt{Madelon} on all several instances of GNA algorithm. For dataset \texttt{Rand}, $A$ is  generated randomly using Gaussian distribution and the vector $y$ is full of ones. Its size is $50 \times 25$. We use $N = \infty$, i.e., $N$ grows linearly with $i$. This experiment highlights the results of Theorem \ref{thm:rate_conv}, i.e., optimal rate and convergence after $d$ iterations of GNA.

The second dataset, \texttt{Madelon} comes from  \citet{guyon2003design}, and posses 2000 data point in dimension 500. Here, we fix $N_{\max}=20$. This experiment illustrates the efficiency of previous methods when used with limited memory. 

We distinguish the case with and without line-search on the $\beta$, the mixing parameter in Anderson acceleration, or the initialization $H_0=\beta I$ in qN methods. Without line-search, the parameter is fixed to $-1$. The starting point $x_0$ is full of ones and $\lambda$ is set such that $\kappa = 10^{-6}$. The experiments are presented in Figure \ref{fig:num_experiments}. Since GMRES and CG are a bit different (the line search is necessary, and they do not  \eqref{eq:linear_algo}), we do not compare them in this section. We performed more advanced numerical experiments in Appendix \ref{sec:num_experiments}.

\begin{figure}
    \centering
    \includegraphics[width=0.45\linewidth]{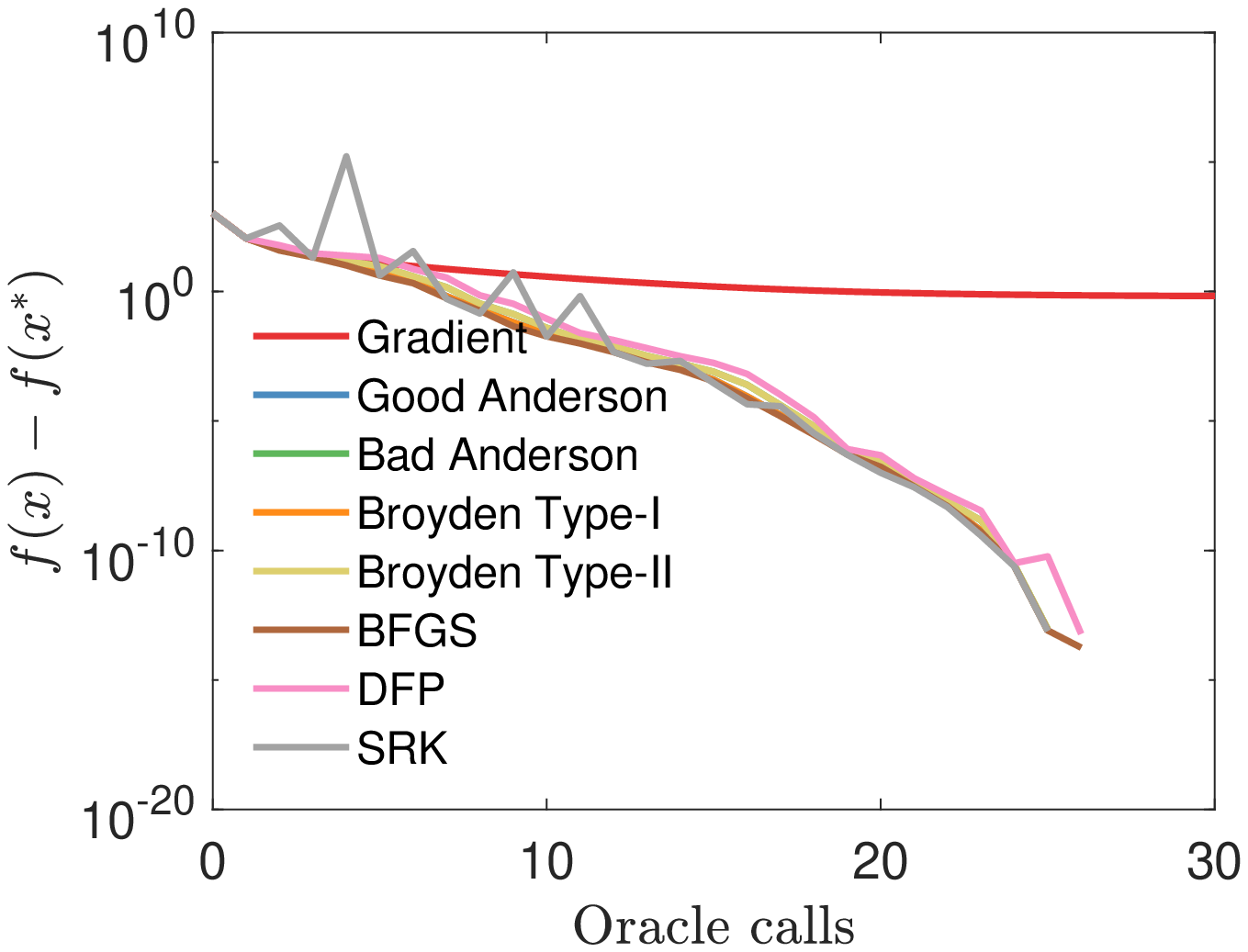}
    \includegraphics[width=0.45\linewidth]{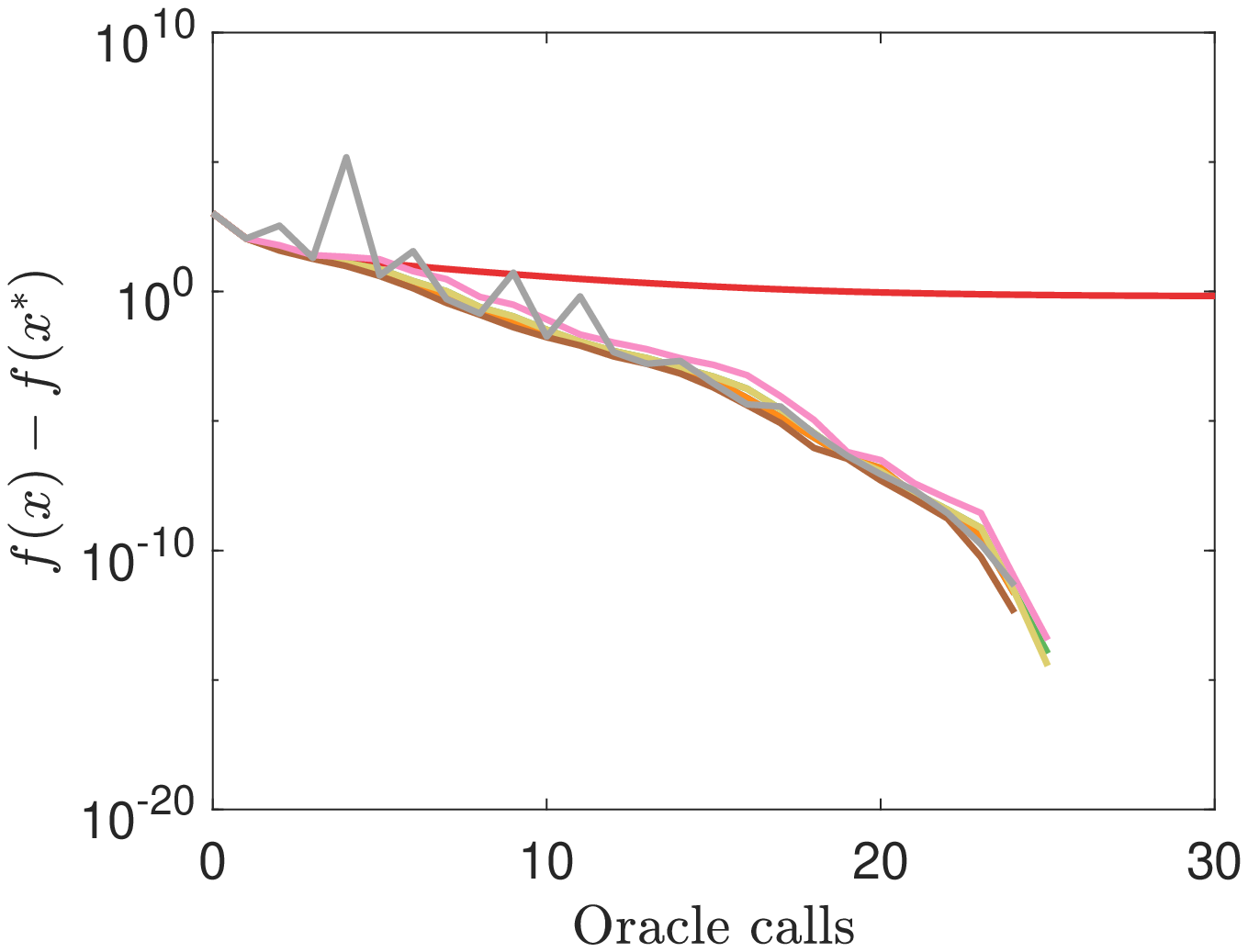}
    \includegraphics[width=0.45\linewidth]{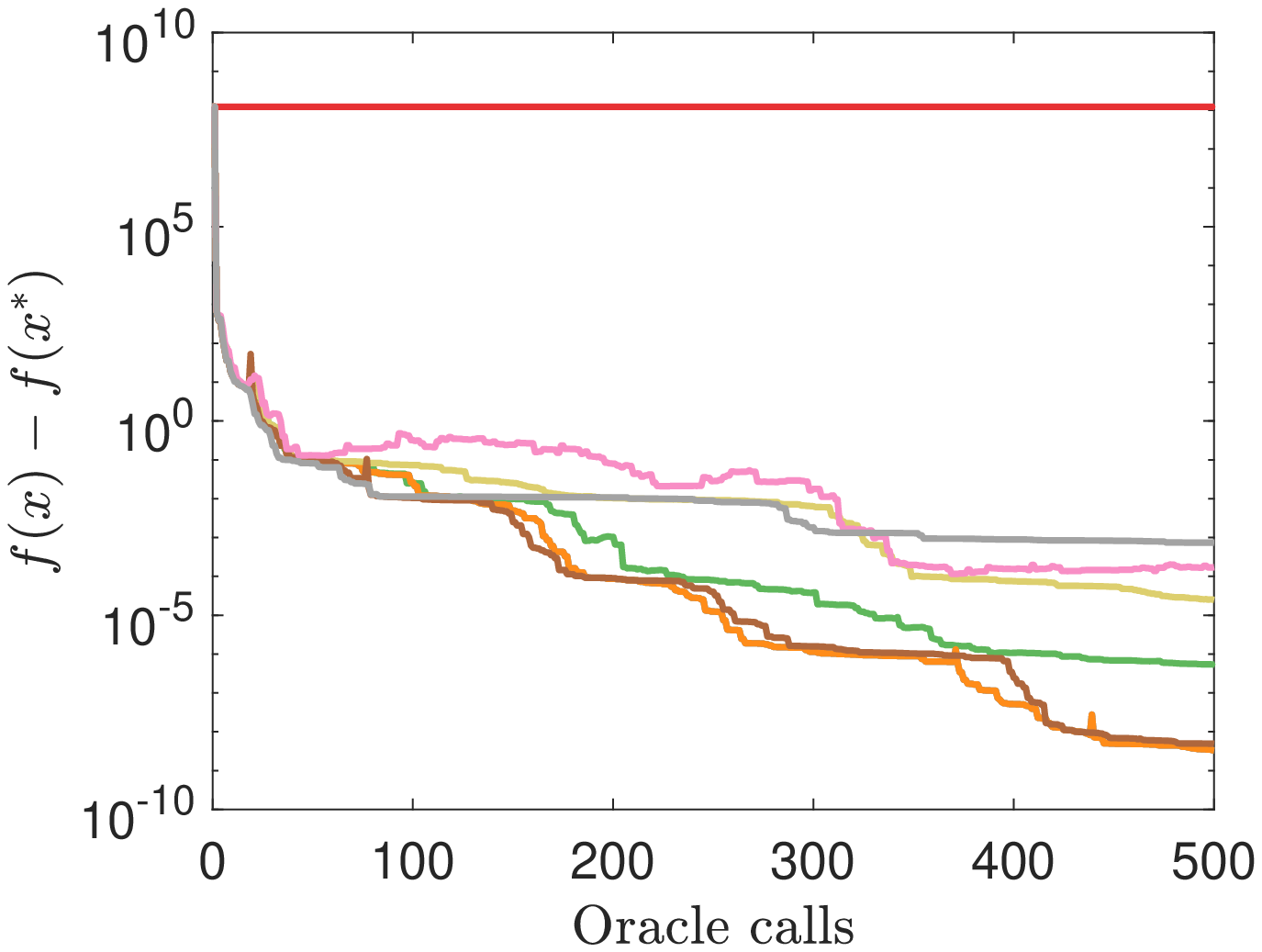}
    \includegraphics[width=0.45\linewidth]{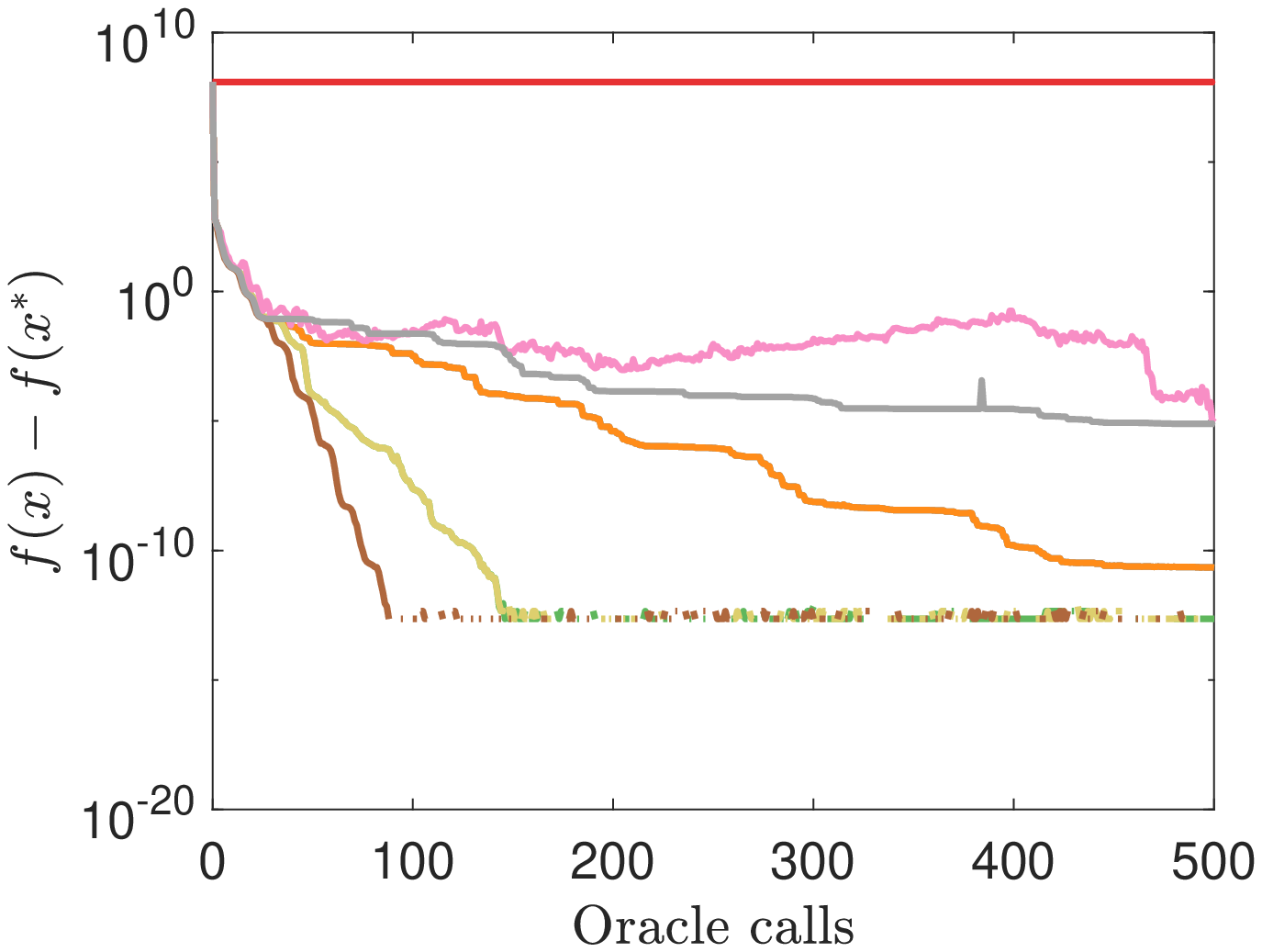}
    \caption{Comparison between different nonlinear acceleration algorithm. Top row: \texttt{Rand}, bottom: \texttt{Sido0}. Left column is with fixed step size, right with line-search.}
    \label{fig:num_experiments}
\end{figure}

We see in Figure \ref{fig:num_experiments} that, when used with full memory, all algorithms behave the same. This is still true when GNA is used with limited memory and without line-search, with a slight advantage to BFGS and Bad Broyden/Anderson method. With line-search, BFGS and Bad Broyden/Anderson methods are superior to the others, which explains why BFGS is more used in practice.


\clearpage
\bibliography{biblio}
\bibliographystyle{plainnat}

\clearpage

\appendix

\section{Missing Proofs}

\subsection{Proof of Proposition \ref{prop:sol_good_anderson}}
\label{prop:sol_good_anderson_proof}
\begin{proposition} 
Let the symmetric positive definite matrix $W$ satisfies \eqref{eq:cond_inv_w}. Then, the coefficients $\gw$ defined in Algorithm \ref{algo:gna} can be computed using the formula
\BEQ
    \gw = \frac{(Y^TR)^{-1}\textbf{1}_N}{\textbf{1}_N^T(Y^TR)^{-1}\textbf{1}_N}.
\EEQ
\end{proposition}
\begin{proof}
We start with \eqref{eq:gw}. If we expand the objective function we have
\[
    \gamma^T (Y-X^*)(G-I)^{-1}(Y-X^*) \gamma.
\]
However, $\gamma^TX^* = x^*$ since $\gamma^T\textbf{1} = 1$. If we drop the constant terms, the objective function becomes
\[
    \gamma^T Y^T(G-I)Y \gamma + 2\gamma^TY(G-I)x^*.
\]
Now, consider the Lagrangian function $\mathcal{L}(\gamma,\pi) $,
\[
    \gamma^T Y^T(G-I)Y \gamma + 2\gamma^TY(G-I)x^* + 2\pi(\gamma^T\textbf{1}-1).
\]
Its derivative over $\gamma$ gives
\[
    \frac{\nabla_\gamma \mathcal{L}(\gamma,\pi)}{2} = Y^T(G-I)Y \gamma + Y(G-I)x^* + \pi\textbf{1}
\]
Using $X^*\gamma=x^*$, we can simplify it into
\[
    \frac{\nabla_\gamma \mathcal{L}(\gamma,\pi)}{2} = Y^T(G-I)(Y-X^*) \gamma + \pi\textbf{1}.
\]
Finally, since $(G-I)^{-1}(Y-X^*) = R$ we have
\[
    \frac{\nabla_\gamma \mathcal{L}(\gamma,\pi)}{2} = Y^TR\gamma + \pi\textbf{1}.
\]
If we equal the derivative to zero,
\[
    \gamma = -\pi(Y^TR)^{-1}\textbf{1}.
\]
Finally, using the constraint that $\gamma^T\textbf{1}=1$, we have
\[
    \pi = \frac{-1}{\textbf{1}^T(Y^TR)^{-1}\textbf{1}}.
\]
\end{proof}

This proposition can be easily extended to the general case there $W=(G-I)^{-1}W_1+W_2$, assuming $(G-I)^{-1}W_1+W_2$ symmetric positive definite. In this case,
\BEQ
    \gamma_{W} = \frac{(Y^T W_1 R + R^TW_2R)^{-1}\textbf{1}}{\textbf{1}^T(Y^T W_1 R + R^TW_2R)^{-1}\textbf{1}}. \label{eq:sol_good_anderson_extended}
\EEQ

\subsection{Proof of Theorem \ref{thm:rate_conv}}
\begin{theorem}\label{thm:rate_conv_proof}
Let $\{(x_1,y_0),\ldots, (x_N,y_{N-1})\}$ be pairs generated by \eqref{eq:linear_algo} (or equivalently \eqref{eq:poly_algo}) where the matrix $G$ is symmetric and $\|G\|\leq 1-\kappa$. Let $\yex$ be the output of Algorithm \ref{algo:gna}. Consider the residual of the extrapolation
\[
    \rex = g(\yex)-\yex.
\]
Then, the norm $\|\rex\|_W$ is bounded by
\BEAS
    \| \rex \|_{W} \leq \| I-(G-I)P \|_2 \frac{\xi^{i-1}}{1+\xi^{2(i-1)}} \| r_1 \|_{W}
\EEAS
where $\xi$ is defined in \eqref{eq:cheby_rate} and $W$ is a positive definite.
\end{theorem}
\begin{proof}
We start with the extrapolation step
\[
    \yex = (Y-PR)\gw.
\]
Since for any vector $v$ such that $\textbf{1}^Tv = 1$ we have
\BEQ
    \textstyle X^*v = (\sum_jv_j) x^* = x^* \label{eq:xstar_ones},
\EEQ
combining \eqref{eq:xstar_ones} with \eqref{eq:link_residual_matrix} gives
\BEQ
    \yex-x^* = ((G-I)^{-1}-P)R\gw. \label{eq:thm1temp1}
\EEQ
In addition, the residual $\rex$ can be written
\BEQ
    \rex = g(\yex)-\yex = (G-I)(\yex-x^*). \label{eq:thm1temp2}
\EEQ
Combining equations \eqref{eq:thm1temp1} and \eqref{eq:thm1temp2} gives
\[
    \rex = (I-(G-I)P)R\gw.
\]
Its norm becomes
\BEAS
    \|\rex\|_W & = & \|(I-(G-I)P)R\gw\|_W  \\
    & \leq & \|(I-(G-I)P)\|_2\|R\gw\|_W.
\EEAS
By definition, 
\[
    \gamma_W  = \argmin_\gamma \|R\gamma\|_W \quad \text{subject to} \;\; \gamma^T\textbf{1} = 1. 
\]
Using the structure of $R$, we have for some polynomial $p_{\gw}$ that $R\gw = p(G)r_1$. This means
\[
    \|R\gw\|_W = \min_{p:p\in\mathcal{P}_{N-1},p(1)=1} \|p(G)r_1\|_W,
\]
where $\mathcal{P}_{N-1}$ is the space of polynomials of degree at most equal to one. Finally,
\BEAS
    & & \min_{p:p\in\mathcal{P}_{N-1},p(1)=1} \|p(G)r_1\|_W \\
    & \leq & \|r_1\|_W\min_{p:p\in\mathcal{P}_{N-1},p(1)=1} \|p(G)\|_2.
\EEAS
The last term is bounded by \eqref{eq:cheby_rate}, concluding the proof.
\end{proof}

\subsection{Proof of Proposition \ref{prop:online_accel}}

\label{prop:online_accel_proof}

\begin{proposition} \label{prop:invariance_gamma}
Consider the output of Algorithm \ref{algo:gna} after $N$ iterations of \eqref{eq:poly_algo}. If we can ensure
\[
    \yex = (Y-PR)\gamma_W = Y\tilde\gamma + X\gamma_{\tilde W},
\]
then $\yex - x^*= p_{N-1}(G)(y_0-x^*)$, where $p_N$ is a polynomial of degree $N-1$ whose coefficients sum to one.
\end{proposition}
\begin{proof}
First, consider $X\gamma_{\bar W}$. We showed it can be written as a polynomial of degree $1$, whose coefficients sum to one, with nonzero leading coefficient. Since $Y\bar \gamma$ is, for any value of gamma, at most of degree $N-1$, the whole expression still have a nonzero leading coefficient, and thus remains of degree $N$. Finally, because $\bar\gamma$ sum to zero, we end with a polynomial of degree $N$, whose coefficients sum to one.
\end{proof}

\section{Elements of Matrix Theory}

\subsection{Woodbury Matrix Identity} 
\label{sec:woodbury}

The Woodbury matrix identity \citep{woodbury1950inverting} computes $(A+USV)^{-1}$ using a update on $A^{-1}$. The formula is
\BEQ
    (A+USV)^{-1} = A^{-1}-A^{-1}U(S^{-1}+VA^{-1}U)^{-1}VA^{-1}.
\EEQ

\subsection{Symmetric Solution of a System of Linear Equations} 
\label{sec:symmetric_solution}

The results in \citep{don1987symmetric,hua1990symmetric} show the symmetric solution of a system of equation. We give here the statement of Theorem 2 in \cite{don1987symmetric}.

\begin{theorem} (Don, 1987) Consider the system of linear equations
\[
    AX=B, \quad X=X^T.
\]
Then, the system admits a symmetric solution if and only if
\[
    AA^{\thicksim} B=B \qquad \text{and} \qquad AB^T=BA^T.
\]
If such condition is satisfied, the explicit formula for $X$ reads
\[
    A^{\thicksim}B + (I-A^{\thicksim}A)(A^+B)^T + (I-A^{\thicksim}A)^T Z (I-A^{\thicksim}A),
\]
where $Z$ is an arbitrary symmetric matrix ($Z=0$ gives the minimum norm solution) and $A^{\thicksim}$ is a minimum-norm generalized reflexive pseudo-inverse of $A$, which means
\[
    AA^{\thicksim}A = A \qquad , \qquad A^{\thicksim}AA^{\thicksim} = A^{\thicksim} \qquad \text{and} \qquad A^{\thicksim}A = (A^{\thicksim}A)^T.
\]
\end{theorem}

This theorem is useful when solving the optimization problem present in BFGS and DFP. As a corollary, we have that the solution of the transposed system
\[
    XA=B
\]
is given by
\BEA
    X & = & BA^+ + (BA^+)^T(I-AA^+) + (I-AA^+)^TZ(I-AA^+), \label{eq:symmetric_solution}\\
    & = & (BA^+)^T + (I-AA^+)^TBA^+ + (I-AA^+)^TZ(I-AA^+), \nonumber
\EEA
under the condition that
\[
    BA^+A=B \qquad \text{and} \qquad A^TB=B^TA,
\]
where $A^+$ is a generalized reflexive pseudo-inverse,
\BEQ
    AA^{+}A = A \qquad \text{and} \qquad A^{+}AA^{+} = A^{+}. \label{eq:reflexive_pseudo_inverse}
\EEQ
Here, the result is a bit more general, as $A^+$ does not need to be minimal-norm.

\section{Explicit Formulas for Quasi-Newton Methods}

\subsection{Generalized Quasi-Newton Step}

Before analyzing each qN method, we introduce the generalized quasi-Newton step. We are looking for an approximation $H\approx (G-I)^{-1}$, where $H$ satisfies the secant equation
\[
    HRC = YC.
\]
After $N$ iterations usual qN methods perform the step
\[
    \yex = y_{N-1}-Hr_{N}.
\]
Here, we are more interested in the \textit{generalized quasi-Newton} (GqN) step 
\BEQ
    y_{\text{Generalized qN}} = (Y-HR)\gamma, \quad \forall\gamma^T\textbf{1}=1
    \label{eq:generalized_qn2}
\EEQ
The next proposition shows that the choice of $\gamma$ \textit{does not change $y_{GqN}$}.
\begin{proposition}
Let $H$ satisfies the secant equation \eqref{eq:secant_equation}. Then, the generalized qN step \eqref{eq:generalized_qn} is invariant in $\gamma$.
\end{proposition}
\begin{proof}
    Consider two different vectors $\gamma_1$ and $\gamma_2$, whose entries sum to one. Then,
    \BEQ
        (Y-HR)\gamma_1 = (Y-HR)\left(\gamma_2+(\gamma_1-\gamma_2)\right) \label{eq:difference_gamma}
    \EEQ
    Let $c = \gamma_1-\gamma_2$. Since the entries of $\gamma_1$ and $\gamma_2$ sum to one, the entries of $c$ sum to zero. In particular, because $C$ is full column rank, this means there exist a vector $v$ such that $Cv=c$. We can rewrite \eqref{eq:difference_gamma} into
    \BEAS
        (Y-HR)\gamma_1 & = & (Y-HR)\left(\gamma_2+Cv\right)\\
        & = & (Y-HR)\gamma_2+(YC-HRC)v.
    \EEAS
    Since $H$ satisfies the secant equation \eqref{eq:secant_equation}, we have $(YC-HRC)=0$ and
    \[
        (Y-HR)\gamma_1 =  (Y-HR)\gamma_2.
    \]
    Because the choice of $\gamma_1$ and $\gamma_2$ was arbitrary, this prove the proposition.
\end{proof}

In this section, we use this proposition to simplify the qN steps, in particular when $\gamma=\gw$ \eqref{eq:gw} for some $W$. The previous results also hold when $H=J^{-1}$ and $JYC=RC$.

\clearpage 

\subsection{(Multi-Secant) Type-I Broyden}
\label{sec:good_broyden}

\subsubsection{Generalized Multi-Secant Broyden Type-I Update}

The multi-secant Broyden method solves
\[
    J = \argmin_J \| J-J_0 \|_{M} \qquad s.t. \;\; JYC=RC
\]
First, we parametrize $J$ to satisfy the constraints automatically. Let the matrix $P_M$ such that
\[
    P_M YC = YC, \qquad (I-P_M)YC = 0, \qquad P_MM(I-P_M)^T = 0.
\]
For example, we can take
\[
    P_M = YC\left((YC)^TM^{-1}YC\right)^{-1}(YC)^{T}M^{-1}.
\]
Now, without loss off generality, we divide $J$ into
\[
    J = J_1P_M + J_2(I-P_M).
\]
The optimization problem becomes
\[
    J = \argmin_{J=J_1P_M + J_2(I-P_M)} \| (J_1-J_0)P_M + (J_2-J_0)(I-P_M) \|_{M} \qquad s.t. \;\; J_1YC=RC
\]
We can expand the objective into
\BEAS
    &   & \| (J_1-J_0)P_M + (J_2-J_0)(I-P_M) \|_{M}\\
    & = & \tr\left(M^{1/2}\left((J_1-J_0)P_M + (J_2-J_0)(I-P_M)\right)M\left((J_1-J_0)P_M + (J_2-J_0)(I-P_M)\right)^TM^{1/2}\right)
\EEAS
Using the property that $P_MM(I-P_M)^T = 0$,
\BEAS
    \| (J_1-J_0)P_M + (J_2-J_0)(I-P_M) \|_{M} & = & \| (J_1-J_0)P_M \|_M + \| (J_2-J_0)(I-P_M) \|_M
\EEAS
Using the expression of $P_M$ and the constraint $J_1YC=RC$,
\BEAS
    (J_1-J_0)P_M & = & (J_1-J_0)YC\left((YC)^TM^{-1}YC\right)^{-1}(YC)^{T}M^{-1}\\
    & = & (J_1YC-J_0YC)\left((YC)^TM^{-1}YC\right)^{-1}(YC)^{T}M^{-1} \\
    & = & (RC-J_0YC)\left((YC)^TM^{-1}YC\right)^{-1}(YC)^{T}M^{-1}
\EEAS
As this term is constant, we remove it from the optimization problem, which becomes
\BEAS
    J & = & \argmin_{J=J_1P_M + J_2(I-P_M)} \| (J_1-J_0)P_M \|_M + \| (J_2-J_0)(I-P_M) \|_M \qquad s.t. \;\; J_1YC=RC\\
    & = & \argmin_{J=J_1P_M + J_2(I-P_M)} \| (J_2-J_0)(I-P_M) \|_M
\EEAS
Clearly, the optimal solution is obtained when $J_2=J_0$. Finally,
\BEAS
    J & = & J_1P_M + J_2(I-PM)\\
    & = & RC\left((YC)^TMYC\right)^{-1}(YC)^{T}M + J_0(I-PM)\\
    & = & J_0 + (RC-J_0YC)\left((YC)^TM^{-1}YC\right)^{-1}(YC)^{T}M^{-1}
\EEAS
Now we use the Woodbury matrix identity (Appendix \ref{sec:woodbury}),
\BEAS
    J^{-1} & = & J_0^{-1} - J_0^{-1}(RC-J_0YC) + \left((YC)^TM^{-1}YC + (YC)^{T}M^{-1}J_0^{-1}(RC-J_0YC)\right)^{-1}(YC)^{T}M^{-1}J_0^{-1}\\
    & = & J_0^{-1} + (YC-J_0^{-1}RC) + \left( (YC)^{T}M^{-1}J_0^{-1}RC\right)^{-1}(YC)^{T}M^{-1}J_0^{-1}
\EEAS

This gives the generalized multi-secant update for the Type-I Broyden method.

\fbox{\begin{minipage}{0.9\linewidth}
\textbf{Generalized Multi-Secant Type-I Broyden Matrix Update}
\BEAS
    J^{-1} = J_0^{-1} + (YC-J_0^{-1}RC) + \left( (YC)^{T}M^{-1}J_0^{-1}RC\right)^{-1}(YC)^{T}M^{-1}J_0^{-1}
\EEAS
\end{minipage}}

Now, consider the generalized qN step \eqref{eq:generalized_qn}
\[
    y_{\text{Broyden Type-I}} = (Y-J^{-1}R)\gamma.
\]
By Proposition \ref{prop:invariance_gamma}, this step is invariant in $\gamma$. In particular, if $\gamma = \gw$ \eqref{eq:gw} with $W=(G-I)^{-1}M^{-1}J_0^{-1}$ (assuming this matrix symmetric positive definite) we can simplify the qN step. Indeed, using the formula \eqref{eq:sol_good_anderson_extended} we have
\[
    \gamma_{(G-I)^{-1}M^{-1}J_0^{-1}} = \frac{(Y^TM^{-1}J_0^{-1}R)^{-1}\textbf{1}}{\textbf{1}^T(Y^TM^{-1}J_0^{-1}R)^{-1}\textbf{1}}.
\]
Injecting those coefficients into the generalized qN step gives
\[
    (Y-J^{-1}R)\gamma_{(G-I)^{-1}M^{-1}J_0^{-1}} = (Y-J_0 R)\gamma_{(G-I)^{-1}M^{-1}J_0^{-1}}.
\]
Indeed, because by definition $C^T1=0$, the term highlighted in red in the equation below equal zero,
\BEAS
    & & J^{-1}R\gamma_{(G-I)^{-1}M^{-1}J_0^{-1}} \\
    & = & \left(J_0^{-1} + (YC-J_0^{-1}RC) + \left( (YC)^{T}M^{-1}J_0^{-1}RC\right)^{-1}\textcolor{red}{(YC)^{T}M^{-1}J_0^{-1}R}\right)\frac{\textcolor{red}{(Y^TM^{-1}J_0^{-1}R)^{-1}\textbf{1}}}{\textbf{1}^T(Y^TM^{-1}J_0^{-1}R)^{-1}\textbf{1}}.
\EEAS
This step is a special instance of Algorithm \ref{algo:gna} as shown in the following box.

\fbox{\begin{minipage}{0.9\linewidth}
\textbf{Generalized Nonlinear Acceleration: Multi-Secant Type-I Broyden Step}

\begin{itemize}
    \item Set $W = (G-I)^{-1}M^{-1}J_0^{-1}$, compute $\gw$ with formula \eqref{eq:sol_good_anderson_extended} (Assuming $W$ symmetric positive definite).
    \item Set $P=J_0^{-1}$.
\end{itemize}
\end{minipage}}

\subsubsection{Simple Multi-Secant Broyden Type-I Update}

In the standard case where $M^{-1}=I$ and $J_0^{-1}=\beta I$, it gives the following update.

\fbox{\begin{minipage}{0.9\linewidth}
\textbf{Standard Multi-Secant Type-I Broyden Matrix Update}
\BEAS
    J^{-1} = \beta I + (YC-\beta IRC) + \left( (YC)^{T}RC\right)^{-1}(YC)^{T}
\EEAS
\end{minipage}}

In this case, the call to Algorithm \ref{algo:gna} is more straightforward.

\fbox{\begin{minipage}{0.9\linewidth}
\textbf{Generalized Nonlinear Acceleration: Simple Multi-Secant Type-I Broyden Step}

\begin{itemize}
    \item Set $W = (G-I)^{-1}$, compute $\gw$ with proposition \eqref{prop:sol_good_anderson}.
    \item Set $P=\beta I$, where $\beta$ is a nonzero scalar.
\end{itemize}
\end{minipage}}

As the structure of this GNA step matches Proposition \ref{prop:online_accel}, this instance is eligible for online acceleration.

\clearpage 

\subsection{(Multi-Secant) Type-II Broyden}
\label{sec:bad_broyden}

\subsubsection{Generalized Multi-Secant Broyden Type-II Update}

The multi-secant Type-II Broyden update is similar to the Type-I. We start with
\[
    H = \argmin_H \| H-H_0 \|_{M^{-1}} \qquad s.t. \;\; HRC=YC
\]
To be coherent with the units, we use $M^{-1}$ as we estimate the inverse of $(G-I)$. Using the matrix
\[
    P_M = RC((RC)^TMRC)^{-1}(RC)^TM,
\]
we decompose $H$ into
\[
    H = H_1P_M + H_2(I-P_M).
\]
The optimization problem becomes
\[
    H = \argmin_H \| (H_1-H_0)P_M \|_{M^{-1}}+\| (H_2-H_0)(I-P_M) \|_{M^{-1}} \qquad s.t. \;\; HRC=YC
\]
The left term is constant since
\[
    H_1P_M = \underbrace{H_1RC}_{=YC}\left((RC)^TMRC\right)^{-1}(RC)^TM.
\]
The optimization problem is thus optimal when $H_2=H_0$. Finally, the update reads
\[
    H = H_1P_M + H_2(I-P_M) = H_0 + (RC-H_0YC)\left((RC)^TMRC\right)^{-1}(RC)^TM.
\]

\fbox{\begin{minipage}{0.9\linewidth}
\textbf{Generalized Multi-Secant Type-II Broyden Matrix Update}
\BEAS
    H = H_0 + (RC-H_0YC)\left((RC)^TMRC\right)^{-1}(RC)^TM.
\EEAS
\end{minipage}}

Now, consider the generalized qN step \eqref{eq:generalized_qn}. Using this update and $\gamma=\gw$ with $W=M$ gives
\[
    y_{\text{Broyden Type-II}} = \left(Y-H_0R + (RC-H_0YC)\left((RC)^TMRC\right)^{-1}\textcolor{red}{(RC)^TMR}\right)\frac{\textcolor{red}{(R^TWR)^{-1}\textbf{1}}}{\textbf{1}^T(R^TMR)^{-1}\textbf{1}}
\]
Since $C^T\textbf{1}=0$, the term in red simplifies and the updates reads
\[
   y_{\text{Broyden Type-II}} =  (Y-H_0R)\gamma_M.
\]
We summarize in the next box the call to GNA to produce the same step than Broyden Type-II.

\fbox{\begin{minipage}{0.9\linewidth}
\textbf{Generalized Nonlinear Acceleration: Generalized Multi-Secant Type-I Broyden Step}
\begin{itemize}
    \item Set $W = M$, compute $\gw$ with equation \eqref{eq:gw}.
    \item Set $P=H_0$
\end{itemize}
\end{minipage}}

\subsubsection{Simple Multi-Secant Broyden Type-II Update}
In the simple case, $M=I$ and $H_0=\beta I$ where $H_0$ is a nonzero scalar. We summarize the simple Broyden update and its GNA call in the boxes below.

\fbox{\begin{minipage}{0.9\linewidth}
\textbf{Generalized Multi-Secant Type-II Broyden Matrix Update}
\BEAS
    H = \beta I + (RC-\beta YC)\left((RC)^TRC\right)^{-1}(RC)^T.
\EEAS
\end{minipage}}

\fbox{\begin{minipage}{0.9\linewidth}
\textbf{Generalized Nonlinear Acceleration: Simple Multi-Secant Type-II Broyden Step}
\begin{itemize}
    \item Set $W = I$, compute $\gw$ with equation \eqref{eq:gw}.
    \item Set $P=\beta I$
\end{itemize}
\end{minipage}}

\clearpage 
\subsection{(Multi-Secant) DFP}
\label{sec:dfp}

\subsubsection{Generalized Multi-Secant DFP update}

We start with
\[
    J = \argmin_J \| J-J_0 \|_M \qquad s.t. \;\; JYC=RC,\,\, J=J^T.
\]
We parametrize $J$ as follow, using equation \eqref{eq:symmetric_solution},
\[
    J = RC(YC)^+ + (RC(YC)^+)^T(I-YC(YC)^+) + (I-YC(YC)^+)^TZ(I-YC(YC)^+)
\]
In particular, we choose as pseudo-inverse
\[
    (YC)^+ = (YC)^+_M = \left((YC)^TM^{-1}(YC)\right)^{-1}(YC)^TM^{-1}.
\]
Clearly we see that  $(YC)^+_M$ satisfies \eqref{eq:reflexive_pseudo_inverse}. We also simplify the expression by writing
\[
    P_M = YC(YC)^+_M.
\]
This way we ensure that $J$ is symmetric an satisfies the secant equation. The optimization problem becomes
\[
    J = \argmin_J \left\| RC(YC)^+ + (RC(YC)^+)^T(I-P_M) + (I-P_M)^TZ(I-P_M) -J_0 \right\|_M
\]
First, we notice that
\[
    P_M M (I-P_M^T) = 0, \qquad YC^+ = YC^+P_M .
\]
Using these properties and the definition of the weighted Frobenius norm \eqref{eq:weighted_frobenius} we have
\[
    J = \argmin_J \left\| \left(RC(YC)^+ - J_0\right) P_M \right\|_M + \left\| \left( (RC(YC)^+_M)^T + (I-P_M)^TZ -  J_0 \right)(I-P_M)\right\|_M
\]
Using again the same ideas on the right term gives
\[
    J = \argmin_J \left\| \left(RC(YC)^+_M - J_0\right) P_M \right\|_M + \left\| \left((RC(YC)^+)^T P_M^T J_0 \right)(I-P_M)\right\| + \left\| (I-P_M)^T \left(Z -  J_0 \right)(I-P_M)\right\|_W
\]
As $Z$ is the only free parameter, we have $Z = J_0$ and
\BEQ
    J_{\text{DFP}} = RC(YC)^+_M + (RC(YC)^+_M)^T(I-P_M) + (I-P_M)^TJ_0(I-P_M) \label{eq:jdfp}
\EEQ
To invert $J$ we use two times the Woodbury matrix identity (Appendix \ref{sec:woodbury}). First, consider the temporary matrix $T$
\BEAS
    T & = & (RC(YC)^+)^T + (I-P_M)^TJ_0 \\
    & = & J_0 + W^{-1}(YC)\left((YC)^TM^{-1}(YC)\right)^{-1}\left((RC) - J_0(YC) \right)^T
\EEAS
The approximation $J$ becomes
\BEAS
    J & = & T + \left( RC(YC)^+_M-T \right)P_M \\
    & = & T + (RC-TYC)\left((YC)^TM^{-1}(RC)\right)^{-1}(YC)^TM^{-1}
\EEAS
Then, using the Woodbury matrix identity we have
\BEAS
    J^{-1} & = & T^{-1} - T^{-1}(RC-TYC) \left( (YC)^TM^{-1} YC + (YC)^TM^{-1} T^{-1}(RC-TYC) \right)^{-1}(YC)^TM^{-1}T^{-1} \\
    & = & T^{-1} + (YC-T^{-1}RC) \left( (YC)^TM^{-1} T^{-1}RC \right)^{-1}(YC)^TM^{-1}T^{-1}
\EEAS
It remains to compute $T^{-1}$ using again the Woodbury matrix identity,
\BEAS
    T^{-1} = J_0^{-1}+J_0^{-1}M^{-1}(YC)\left((RC)^TJ_0^{-1}M^{-1}(Y C)\right)^{-1}\left((YC) - J_0^{-1}(RC) \right)^{T}.
\EEAS

We summarize the update in the following box.

\fbox{\begin{minipage}{0.9\linewidth}
\textbf{Generalized Multi-Secant DFP Matrix Update}
\BEAS
    J^{-1} = T^{-1} + (YC-T^{-1}RC) \left( (YC)^TM^{-1} T^{-1}RC \right)^{-1}(YC)^TM^{-1}T^{-1}
\EEAS
    where
\BEAS
    T^{-1} & = & J_0^{-1}+J_0^{-1}M^{-1}(YC)\left((RC)^TJ_0^{-1}M^{-1}(Y C)\right)^{-1}\left((YC) - J_0^{-1}(RC) \right)^{T}.
\EEAS
\end{minipage}}

\subsubsection{Simple Multi-Secant DFP update}

In the (much) simpler case where $M=(G-I)^{-1}$ and $J_0^{-1} = \beta I$, we have
\BEAS
    T^{-1} & = & \beta I+(RC)\left((RC)^T(R C)\right)^{-1}\left((YC) - \beta(RC) \right)^{T},\\
    J^{-1} & = & T^{-1} + (YC-T^{-1}RC) \left( (RC)^TT^{-1}RC \right)^{-1}(RC)^TT^{-1}.
\EEAS
We can simplify the expression by looking at $(RC)^TT^{-1}$,
\BEAS
    (RC)^TT^{-1} & = & \beta (RC)^T+(RC)^T(RC)\left((RC)^T(R C)\right)^{-1}\left((YC) - \beta(RC) \right)^{T}\\
    & = & (YC)^T.
\EEAS
The expression of $J^{-1}$ becomes simpler,
\[
    J^{-1} = T^{-1} + (YC-T^{-1}RC) \left( (YC)^TRC \right)^{-1}(YC)^T.
\]
We can simplify even more since
\BEAS
    T^{-1}RC & = & \beta RC+(RC)\left((RC)^T(R C)\right)^{-1}\left((YC) - \beta(RC) \right)^{T}RC\\
    & = & RC\left((RC)^T(R C)\right)^{-1}(YC)^TRC.
\EEAS
Multiplying the equation above by $\left( (YC)^TRC \right)^{-1}(YC)^T$ gives
\BEAS
    T^{-1}RC\left( (YC)^TRC \right)^{-1}(YC)^T & = &  RC\left((RC)^T(R C)\right)^{-1}(YC)^TRC\left( (YC)^TRC \right)^{-1}(YC)^T, \\
    & = & RC\left((RC)^T(R C)\right)^{-1}(YC)^T.
\EEAS
We update the expression of $J^{-1}$ using this result,
\[
    J^{-1} = T^{-1} + YC \left( (YC)^TRC \right)^{-1}(YC)^T - RC\left((RC)^T(RC)^T\right)^{-1}(YC)^T.
\]
If we replace $T^{-1}$ in $J^{-1}$, we have
\BEAS
    J^{-1} & = &  \beta I+(RC)\left((RC)^T(R C)\right)^{-1}\left((YC) - \beta(RC) \right)^{T}\\
           &   & + YC \left( (YC)^TRC \right)^{-1}(YC)^T - RC\left((RC)^T(R C)\right)^{-1}(YC)^T,\\
           & = & \beta \left(I - (RC)\left((RC)^T(R C)\right)^{-1}(RC)\right) +  YC \left( (YC)^TRC \right)^{-1}(YC)^T
\EEAS

This gives the standard DFP matrix update, summarized below.

\fbox{\begin{minipage}{0.9\linewidth}
\textbf{Standard Multi-Secant DFP Matrix Update}
\BEAS
    J^{-1} = \beta \left(I - (RC)\left((RC)^T(R C)\right)^{-1}(RC)^T\right) +  YC \left( (YC)^TRC \right)^{-1}(YC)^T.
\EEAS
\end{minipage}}

The generalized qN step \eqref{eq:generalized_qn} reads
\[
    y_{\text{DFP}} = (Y-J^{-1}_{DFP}R)\gamma
\]
By proposition \ref{prop:invariance_gamma}, the choice of $\gamma$ does not impact the result. We pick in particular
\[
    \gamma = \gamma_{I} = \frac{(R^TR)^{-1}\textbf{1}}{\textbf{1}^T(R^TR)^{-1}\textbf{1}}.
\]
Because $C^T\textbf{1} = 0$, we have
\[
    (RC)^TR\gamma_{I} = C^TR^T\frac{(R^TR)^{-1}\textbf{1}}{\textbf{1}^T(R^TR)^{-1}\textbf{1}} = 0.
\]
Using this relation, the formula of the step becomes simpler,
\[
    y_{\text{DFP}} = Y\gamma_I-\left(\beta I + YC \left( (YC)^TRC \right)^{-1}(YC)^T\right)R\gamma_I
\]
We have a perfect match with the structure of Algorithm \ref{algo:gna}.

\fbox{\begin{minipage}{0.9\linewidth}
\textbf{Generalized Nonlinear Acceleration: Multi-Secant DFP}\\
\begin{itemize}
    \item Set $W=I$ and compute $\gamma_W$ using formula \eqref{eq:gw}.
    \item Set $P = \beta I + YC \left( (YC)^TRC \right)^{-1}(YC)^T $, where $\beta$ is a nonzero scalar.
\end{itemize}
\end{minipage}}

\paragraph{About the choice of $W$.} We had two possible choices of $W$ to simplify the expression of the qN step. For example, $W=(G-I)^{-1}$ leads to
\BEAS
    J^{-1}_{DFP}R\gamma_{(G-I)^{-1}} & = & \beta \left(I - (RC)\left((RC)^T(R C)\right)^{-1}(RC)^T\right)R\gamma_{(G-I)^{-1}},\\
    & = & \beta Rv,\qquad v = \left(I - C\left((RC)^T(R C)\right)^{-1}(RC)^T(RC)^T \right)\gamma_I
\EEAS
The major problem with this simplification is that it is unclear if the coefficient $v_N$, associated with $R_N$, in non-zero. This means we potentially lose the structure of \eqref{eq:poly_algo} as we do not satisfy the assumptions of Proposition \ref{prop:online_accel}.

\clearpage
\subsection{(Multi-Secant) BFGS}
\label{sec:bfgs}

The Multi-Secant BFGS update follows the same reasoning that DFP method. It suffices to take the update of $J_{DFP}$ \eqref{eq:jdfp}, swap the matrices $Y$ and $R$ and replacing $M^{-1}$ by $M$ and this gives the multi-secant BFGS update, summarized in the box below.

\fbox{\begin{minipage}{0.9\linewidth}
\textbf{Generalized Multi-Secant BFGS Matrix Update}
\BEAS
     H = YC(RC)^+_M + (YC(RC)^+_M)^T(I-P_M) + (I-P_M)^TH_0(I-P_M) 
\EEAS
where
\BEAS
    (RC)_M^+ & = & \left((RC)^TMRC\right)^{-1}(RC)^TM \\
    P_M & = & RC(RC)_M^+
\EEAS
\end{minipage}}

\subsubsection{Simple Multi-Secant DFP update}

Usually, the BFGS update is used with parameters $M=(G-I)^{-1}$ and $H_0=\beta I$ where $\beta$ is a nonzero scalar. We have the following simplifications.
\BEAS
    (RC)_M^+ & = & \left((YC)^TRC\right)^{-1}(YC)^T, \\
    P_M & = & RC\left((YC)^TRC\right)^{-1}(YC)^T.
\EEAS
In particular, the term $(YC(RC)^+_M)^T(I-P_M)$ is equal to zero as
\BEAS
    (YC(RC)^+_M)^T(I-P_M) & = & YC\left((YC)^TRC\right)^{-1}\textcolor{red}{(YC)^T} \left( I- \textcolor{red}{RC\left((YC)^TRC\right)^{-1}}(YC)^T\right), \\
    & = & 0.
\EEAS
The update becomes
\BEAS
    H & = & YC(RC)^+_M + \beta (I-P_M)^T(I-P_M),\\
      & = & \beta (I-P_M)^T + \left(YC-\beta RC\right)(RC)^+_M
\EEAS
Replacing $(RC)_M^+$ and $P_M$  gives the "standard" multi-secant BFGS update.

\fbox{\begin{minipage}{0.9\linewidth}
\textbf{Standard Multi-Secant BFGS Matrix Update}
\BEAS
     H = \beta \left(I-YC\left( (YC)^TRC \right)^{-1}(RC)^T\right) + \left(YC-\beta RC\right)\left((YC)^TRC\right)^{-1}(YC)^T
\EEAS
\end{minipage}}

This time, we use the parameters $\gw$ \eqref{eq:gw} in the generalized qN step \eqref{eq:generalized_qn} with $W=(G-I)^{-1}$. This simplifies the qN step since
\[
    (YC)^TR\gamma_{(G-I)^{-1}}=C^TY^TR\frac{(Y^TR)^{-1}\textbf{1}}{\textbf{1}^T(Y^TR)^{-1}\textbf{1}} = 0.
\]
Using this result, \eqref{eq:generalized_qn} becomes
\[
    (Y-H_{\text{BFGS}}R)\gamma_{(G-I)^{-1}} =  \left(Y-\beta \left(I-YC\left( (YC)^TRC \right)^{-1}(RC)^T\right)R\right)\gamma_{(G-I)^{-1}}.
\]
Again, this matches perfectly the structure of Algorithm \eqref{algo:gna}, whose parameters are summarized below.

\fbox{\begin{minipage}{0.9\linewidth}
\textbf{Generalized Nonlinear Acceleration: Multi-Secant BFGS}\\
\begin{itemize}
    \item Set $W=(G-I)^{-1}$ and compute $\gamma_W$ using Proposition \eqref{prop:sol_good_anderson}.
    \item Set $P = \beta \left(I-YC\left( (YC)^TRC \right)^{-1}(RC)^T\right) $, where $\beta$ is a nonzero scalar.
\end{itemize}
\end{minipage}}

\clearpage
\subsection{SR-\texorpdfstring{$1$}{1} and SR-\texorpdfstring{$k$}{k}}
\label{sec:srk}

Here, we want to update the matrix $H_0$ with a symmetric matrix, with the lowest possible rank, so that the update $H$ satisfies the secant equation \eqref{eq:secant_equation}. In other words, we want to solve
\[
    H = \argmin_H \textbf{rank} \left( H-H_0 \right), \quad HRC=YC, \quad (H-H_0) = (H-H_0)^T.
\]
We can writte the above problem as follow,
\[
    H = \argmin_H \textbf{rank} \left( \Delta \right), \quad \Delta RC=YC-H_0RC, \quad \Delta = \Delta^T.
\]
where $\Delta = H-H_0$. Using the solution to a symmetric system \eqref{eq:symmetric_solution}, we get
\[
    \Delta =  \left(YC-H_0RC\right)(RC)^+ + (\left(YC-H_0RC\right)(RC)^+)^T(I-(RC)(RC)^+) + (I-(RC)(RC)^+)^TZ(I-(RC)(RC)^+),
\]
Clearly, we can easily reduce the rank if $Z = 0$. It remains to find the right pseudo inverse $(RC)^+$ which minimizes the rank of
\[
    \Delta =  \left(YC-H_0RC\right)(RC)^+ + (\left(YC-H_0RC\right)(RC)^+)^T(I-(RC)(RC)^+).
\]
In particular, choosing 
\[
    (RC)^+ = \left( \left(YC-H_0RC\right)^TRC \right)^{-1}\left(YC-H_0RC\right)^T
\]
leads to the symmetric rank $k$ update
\[
    \Delta =  \left(YC-H_0RC\right)\left( \left(YC-H_0RC\right)^TRC \right)^{-1}\left(YC-H_0RC\right)^T.
\]
This leads to the SR-$k$ update below.

\fbox{\begin{minipage}{0.9\linewidth}
\textbf{Symmetric Rank-$k$ Update}
\BEAS
     H = H_0 + \left(YC-H_0RC\right)\left( \left(YC-H_0RC\right)^TRC \right)^{-1}\left(YC-H_0RC\right)^T.
\EEAS
\end{minipage}}

The sr-$k$ step reads
\[
    y_{\text{sr}-k} = \left(Y-\left(H_0 + \left(YC-H_0RC\right)\left( \left(YC-H_0RC\right)^TRC \right)^{-1}\left(YC-H_0RC\right)^T\right)R\right)\gamma.
\]
Using
\[
    \gamma_{[(G-I)^{-1}+H_0]} =\frac{ \left(R^T\left((G-I)^{-1}+H_0\right)R\right)^{-1}\textbf{1}}{\textbf{1}^T\left(R^T\left((G-I)^{-1}+H_0\right)R\right)^{-1}\textbf{1}}
\]
we have
\[
    HR\gamma_{[(G-I)^{-1}+H_0]} = H_0R\gamma_{[(G-I)^{-1}+H_0]}.
\]
Notice that $\gamma$ can be simplified using \eqref{eq:sol_good_anderson_extended},
\[
    \gamma_{[(G-I)^{-1}+H_0]} =\frac{ \left(Y^TR-R^TH_0R\right)^{-1}\textbf{1}}{\textbf{1}^T\left(Y^TR-R^TH_0R\right)^{-1}\textbf{1}}.
\]
The call to GNA is straightforward and summarized in the box below.

\fbox{\begin{minipage}{0.9\linewidth}
\textbf{Generalized Nonlinear Acceleration: Symmetric Rank-$k$}
\begin{itemize}
    \item Set $W = (G-I)^{-1}+H_0$ and compute $\gw$ using \eqref{eq:sol_good_anderson_extended}.
    \item Set $P = H_0$ where $H_0$ is symmetric and $(G-I)\succ H_0$.
\end{itemize}
\end{minipage}}

\section{Explicit Formulas for Krylov-subspace Methods}

\subsection{GMRES}
\label{sec:gmres}

In the case of GMRES (or equivalently MINRES, since we work with symmetric matrices), the iterations are \textit{not} coupled. The algorithm creates a smart basis $K_N$ of $\mathcal{K}_N$ using Arnoldi (MINRES, \cite{paige1975solution}) or Lancoz (GMRES, \cite{saad1986gmres}) then computes
\[
    y_{\text{GMRES}} = \argmin_{x\in x_0 + \mathcal{K}_{N-1}} \| x-g(x) \|_2.
\]
Here, we see that the iterates belongs to $\mathcal{K}_{N-1}$ rather than $\mathcal{K}_{N}$.

In this section, we show that $y_{\text{GMRES}} = y_{\text{GNA}}$ when $W=I$ and $P=0$. In particular, we show that
\BEQ
    y_{\text{GMRES}} = x^* + p^*(G)(x_0-x^*) \quad \text{and} \quad y_{\text{GNA}} = x^* + p^*(G)(x_0-x^*), \quad \text{where} \;\; p^* = \argmin_{\substack{p : p(1)=1,\\ \deg(p) \leq N}} \| p(G)r_1)\|_2. \label{eq:link_gmres_gna}
\EEQ
\subsubsection{GMRES and \texorpdfstring{$p^*$}{p*}}
First, we start with the definition of the GMRES iterate. Indeed,
\[
    y_{\text{GMRES}} = \argmin_{x\in x_0 + \mathcal{K}_{N-1}} \| x-g(x) \|_2 = \argmin_{x\in x_0 + \mathcal{K}_{N-1}} \| (G-I)(x-x^*) \|_2
\]
Using \eqref{eq:krylov_subspace}, we have
\BEA
    y_{\text{GMRES}} = x_0 + q^*(G)r_1, \;\; q^* & = & \argmin_{q:\deg(q) \leq N-2} \| (G-I)(x_0+q(G)r_0-x^*) \|_2 \label{eq:ygmres} \\
    & = & \argmin_{q:\deg(q) \leq N-2} \| (I+(G-I)q(G))r_0 \|_2\nonumber
\EEA
Instead of optimizing over all polynomial $q$ of degree $\leq N-2$, we optimize over $p$ of degree at most $N-1$ whose coefficients sum to one since
\[
    \min_{\substack{
    p=(I+(G-I)q(G)),\\\deg(q) \leq N-1
    }} \| p(G)r_1)\|_2 = \min_{\substack{p : p(1)=1,\\ \deg(p) \leq N}} \| p(G)r_1)\|_2.
\]
Let $p^*$ be the optimal polynomial. We can deduce $q^*$ from $p^*$ using
\[
    q^*(G) = (G-I)^{-1}(p^*(G) - I)
\]
If we replace $q^*$ in \eqref{eq:ygmres} by the expression above, we have
\BEAS
    y_{\text{GMRES}} & = & x_0 +  (G-I)^{-1}(p^*(G) - I)r_1,\\
    & = & x^* + (G-I)^{-1}p^*(G)r_1,\\
    & = & x^* + p^*(G)(G-I)^{-1}r_1,\\
    & = & x^* + p^*(G)(x_0-x^*).
\EEAS
This prove the first part of \eqref{eq:link_gmres_gna}.

\subsubsection{GNA and \texorpdfstring{$p^*$}{p*}}
We start with the definition of the GNA iterate when $W=I$ and $P = 0$,
\BEQ
    y_{\text{GNA}} = (Y-PR) \gamma_W = Y\gamma_I, \qquad \gamma_I = \argmin_{\gamma : \gamma^T\textbf{1} = 1} \| R\gamma \|_2. \label{eq:gna_gmres_iteration}
\EEQ
Since $R$ is a basis of $\mathcal{K}_N$, we have that $R\gamma = p(G)r_1$. In addition, because $r_i$, the i-th column of $R$, can be written as
\[
    r_i = (G-I)(y_i-x^*) = (G-I)p_i(G)(x_0-x^*) = p_i(G) r_1,
\]
where $p_i(G)$ is defined in \eqref{eq:poly_algo} and $p_i(1) = 1$, we have that for all $p :\deg(p)\leq N-1,\,  p(1) = 1$, the exist coefficients $\gamma : \gamma^T1 = 1$ such that
\[
    p(G) = \sum_{i=1}^{N} \gamma_i p_i(G) r_1 = \sum_{i=1}^{N} \gamma_ir_i = R\gamma.
\]
By consequence,
\[
    R\gamma_I = p^*(G)r_1, \qquad \gamma_I = \argmin_{\gamma : \gamma^T\textbf{1} = 1} \| R\gamma \|_2, \qquad p^* = \argmin_{p:\deg(p)\leq N-1, p(1)=1} \| p(G)r_1 \|_2
\]
We can inject this solution in \eqref{eq:gna_gmres_iteration},
\BEAS
    y_{\text{GNA}} & = & Y\gamma_I,\\
    & = & x^* + (G-I)^{-1}R\gamma_I,\\
    & = & x^* + (G-I)^{-1}p^*(G)r_1,\\
    & = & x^* + p^*(G)(x_0-x^*).
\EEAS
This prove the second part of \eqref{eq:link_gmres_gna}.

\subsection{Conjugate Gradient}
\label{sec:cg}

TBA

\section{Numerical Experiments}
\label{sec:num_experiments}
TBA

\end{document}